\documentclass[11pt]{amsart}
\usepackage{amssymb}
\usepackage{amsmath}
\usepackage{amsthm}
\usepackage{empheq}
\usepackage{tikz}
\usepackage{float}
\usetikzlibrary{calc}
\usetikzlibrary{decorations.pathreplacing}
\usetikzlibrary{patterns}
\usepackage{bm}
\usepackage{comment}
\usepackage{mathrsfs}
\usepackage{tikzscale}
\usepackage[T1]{fontenc}
\usepackage[utf8]{inputenc}
\usepackage{todonotes}
\usepackage{hyperref}
\let\svthefootnote\thefootnote
\newcommand\freefootnote[1]{%
	\let\thefootnote\relax%
	\footnotetext{#1}%
	\let\thefootnote\svthefootnote%
}

\newtheorem{theorem}{Theorem}[section]

\newtheorem{lemma}[theorem]{Lemma}

\theoremstyle{definition}
\newtheorem{definition}[theorem]{Definition}

\newtheorem{remark}[theorem]{Remark}

\numberwithin{equation}{section}

\everymath{\displaystyle}

\newcommand*\mR{\mathbb{R}}
\newcommand*\mP{\mathbb{P}}
\newcommand*\mE{\mathbb{E}}

\newcommand*\supp{{\rm supp}}

\newcommand*\eps{\varepsilon}

\newcommand*\I{\textnormal{I}}
\newcommand*\II{\textnormal{II}}
\usepackage[left = 2cm, right = 2cm, top = 2.5cm, bottom = 2.5cm]{geometry}
\begin{document}

\title{Equivalence of definitions of fractional caloric functions}

\author{Artur Rutkowski}
\email{artur.rutkowski@pwr.edu.pl}
\address{Faculty of Pure and Applied Mathematics,
	Wroc\l aw University of Science and Technology,
	Wyb. Wyspia\'nskiego 27, 50-370 Wroc\l aw, Poland}
\subjclass[2020]{Primary 35S16, 35C15, 35D30}
\keywords{fractional Laplacian, caloric function, {caloric measure}, heat equation, distributional solution}

\begin{abstract}
	We prove equivalence between nonnegative distributional solutions of the fractional heat equation and caloric functions, i.e., functions satisfying the mean value property with respect to the space-time isotropic $\alpha$-stable process. We also provide sufficient conditions for the boundary and exterior data under which the solutions are classical and we give off-diagonal estimates for the derivatives of the Dirichlet heat kernel and the lateral Poisson kernel, which might be of their own interest.
\end{abstract}
\maketitle
\section{Introduction}
Let $d\geq 2$ and $\alpha\in (0,2)$. For functions $u\colon \mR^d \to \mR$ and $x\in \mR^d$ we define
\begin{align*}
	\Delta^{\alpha/2}u(x) = \int_{\mR^d} (u(y) - u(x) - \nabla u(x) (y-x) \textbf{1}_{B_1(x)}(y)) \nu(y-x)\, dy,
\end{align*}
where $\nu(y) = c_{d,\alpha} |y|^{-d-\alpha}$ and $B_r(x) = B(x,r)$, $r>0$, is the Euclidean ball in $\mR^d$. The above integral converges, e.g., for $u\in C^2_b(\mR^d)$. 

The goal of the paper is to investigate the connection between three notions of solution to the fractional heat equation in a bounded open set $D\subset \mR^d$:
\begin{align}\label{eq:FHE}
	\begin{cases}
		\partial_tu(t,x) = \Delta^{\alpha/2}u(t,x),\quad &t\in (0,T),\ x\in D,\\
		u(t,x) = g(t,x),\quad &t\in (0,T),\ x\in D^c,\\
		u(0,x) = u_0(x),\quad &x\in D.
	\end{cases}
\end{align}
The first definition is formulated as a mean value property with respect to the space-time isotropic $\alpha$-stable process $\dot{X}_t :=(-t,X_t)$, where $X_t$ is the isotropic $\alpha$-stable process in $\mR^d$. For a Borel set $G\subset \mR^{d+1}$ we let $\tau_G = \inf\{t>0: \dot{X}_t \notin G\}$\footnote{We use the same notation for the first exit times of $X_t$, i.e., for Borel $D\subset \mR^d$, $\tau_D = \inf\{t>0: X_t\notin D\}$.}.
\begin{definition}
	Let $T>0$. We say that $u\colon (0,T)\times \mR^d\to [0,\infty)$ is caloric in $(0,T)\times D$, if for every $(t,x)\in (0,T)\times D$,
	\begin{equation}\label{eq:para}
		u(t,x) = \mE^{(t,x)} u(\dot{X}_{\tau_{G}}) <\infty,
	\end{equation}
	for every open $G\subset\subset (0,T)\times D$ such that $(t,x)\in \overline{G}$. We say that $u\colon [0,T)\times \mR^d\to [0,\infty)$ is caloric in $[0,T)\times D$ if the same holds for every open $G\subset\subset [0,T)\times D$.
\end{definition}

We stress that in the above definition we only consider nonnegative functions.

For the classical heat equation, the approach via the space-time Brownian motion is discussed in the book of Doob \cite{MR1814344}, see also \cite{MR833742,MR2907452}. Fractional caloric functions defined as above were studied by Chen and Kumagai \cite{MR2008600}, and more recently the structure of caloric functions was investigated in the papers of Chan, G\'omez-Castro, and V\'azquez \cite{MR4462819} and Armstrong, Bogdan, and the author~\cite{HK}. Theorem~6.3 in \cite{HK} states that if $D$ is a bounded Lipschitz set, then any function $u$ caloric in $[0,T)\times D$ equal to $g$ on $[0,T)\times D^c$ can be decomposed into three parts:
\begin{align}\label{eq:urepr}
	u(t,x) = P_t^D u_0(x) + \int_0^t\int_{D^c} J^D(t,x,s,z)g(s,z)\, dz\, ds + \int_{[0,t)}\int_{\partial D} \eta_{t-s,Q}(x)\, \mu(dQds),
\end{align}
where $\mu$ is a nonnegative Radon measure on $[0,T)\times \partial D$ and $P_t^D$, $J^D$, and $\eta$ are, respectively, the Dirichlet semigroup, the lateral Poisson kernel, and the parabolic Martin kernel of $D$ for $\Delta^{\alpha/2}$, see Section~\ref{sec:prelim} for details. For $\alpha$-$harmonic$ functions the Poisson representation immediately yields smoothness, which implies that the functions satisfy the fractional Laplace equation. This is not the case for caloric functions: there are examples which are only H\"older in time \cite[Remark~5.12]{HK}, see also the counterexample by Chang-Lara and D\'{a}vila \cite[Example~2.4.1]{MR3115838} given for viscosity solutions of \eqref{eq:FHE}, but also relevant for caloric functions. In these examples, either the \textit{exterior condition} $g$ or the \textit{boundary condition} $\mu$ is strongly discontinuous in time. 

As one of the main results of this paper, we give a sufficient condition on $g$ and $\mu$ under which a caloric function is a classical solution of the fractional heat equation.
\begin{definition}
	We say that $u\colon [0,T]\times \mR^d \to [0,\infty)$ is a classical solution to \eqref{eq:FHE} if $u\in C((0,T),L^1(1\wedge \nu))$, $u\in C((0,T)\times D)$, $\partial_t u,$ $Du$ and $D^2u$ exist and are continuous in $(0,T)\times D$ and \eqref{eq:FHE} is satisfied pointwise.
\end{definition}
{Let $P_D$ be the Poisson kernel of $D$, see \eqref{eq:PDGD} below for the definition. In the formulation of the next result, we use $L^1$ spaces weighted with $D^c\ni z\mapsto P_D(x_0,z)$ for some fixed $x_0\in D$. Since $D$ is Lipschitz, this function is a probability density (hence nonnegative and integrable), blows up as $z\to \partial D$, and behaves like $1\wedge \nu(z)$ away from $\partial D$, see  \cite{Chen1998,MR1991120} for the sharp estimates.}
\begin{theorem}\label{th:mvpclass}
	Let $D$ be bounded and Lipschitz and assume that $u\geq 0$ is caloric in $[0,T)\times D$ with the representation \eqref{eq:urepr}. If 
	 {$g\in C^{\rm Dini}_{\rm loc}((0,T),L^1(D^c,P_D(x_0,\cdot)))$ for some $x_0\in D$}, and $\mu\in C^{\rm Dini}_{{\rm loc}}((0,T),\mathcal{M}(\partial D))$, then $u$ is a classical solution to \eqref{eq:FHE}.
\end{theorem}
Here, $\mathcal{M}(\partial D)$ is the space of finite signed Borel measures on $\partial D$ with the total variation norm. The proof of Theorem~\ref{th:mvpclass} is given in Section~\ref{sec:mvpclass} below.

{The Dini condition emerges in a rather natural way in the proof, as we need to compensate for the singularity of $\partial_t J^{B_1(0)}(t,x,s,z)$, which blows up roughly as $(t-s)^{-1}$ when $s\nearrow t$, see Theorem~\ref{th:JBest} below. Related results and arguments for elliptic equations were given by Burch \cite{MR521856} for second order operators, and Grzywny, Kassmann, and Leżaj \cite{MR4194536} for nonlocal operators -- both papers gave Dini-type conditions on data that are sufficient to get $C^2$ solutions. We were unable to locate exact counterparts of our results in the context of second order parabolic equations, but the solutions obtained as the double layer heat potentials are classical for continuous boundary data, see, e.g., \cite[Theorem~2.3]{MR2907452}, which indicates that it might be possible to weaken the assumptions slightly in our case, too. The majority of the proof of Theorem~\ref{th:mvpclass} consists of establishing new estimates for $J^{B_1}$ and its derivatives, and obtaining the necessary regularity of caloric functions.}

 
The third class of solutions studied here are the distributional solutions. We note that the following definition includes all classical solutions by, e.g., \cite[Lemma~3.3]{MR1671973}.
\begin{definition}
	We say that $u\colon [0,T]\times \mR^d \to [0,\infty)$ is a distributional solution to \eqref{eq:FHE}, if for every $\phi\in C_c^\infty([0,T)\times D)$, 
	\begin{align}\label{eq:weakform}
		\int_D \phi(t,x)u(t,x)\, dx = \int_D \phi(s,x)u(s,x)\, dx + \int_s^t \int_{\mR^d} (\partial_t + \Delta^{\alpha/2}) \phi(\tau,x) u(\tau,x) \, dx\, d\tau,\quad 0\leq s<t<T,
	\end{align}
	and the integrals converge absolutely.
\end{definition}
The second main result of the paper states that being a (nonnegative) caloric function is equivalent to being a distributional solution to \eqref{eq:FHE}, see Theorems~~\ref{th:mvptodist} and \ref{th:disttomvp} below. The analogue for the Poisson equation was given by Bogdan and Byczkowski \cite{MR1671973}, see also Chen \cite{MR2515419} and \cite{MR4194536} for more general results concerning elliptic equations. For the fractional heat equation the authors of \cite{MR4462819} show that in $C^{1,1}$ domains, given that $u_0,g,$ and $\mu$ are regular enough, the functions of the form \eqref{eq:urepr} satisfy a specific weak-dual formulation, which ensures uniqueness of solutions. Our equivalence result, together with \cite{HK}, gives a complete classification of nonnegative distributional solutions to \eqref{eq:FHE} in Lipschitz sets, which seems to be new even when $D$ is smooth. We refer, e.g., to \cite{MR3211862,MR3278937,MR4668351} for other representation results in this vein.
 {The general ideas of proofs of Theorems~\ref{th:mvptodist} and \ref{th:disttomvp} follow those used in the elliptic case. The calculations are slightly different and more involved, as we need to handle the additional time variable. We also use some estimates of $p_t^D$ and $J^D$ obtained recently in \cite{HK}.} 

In Section~\ref{sec:estimates} we establish off-diagonal estimates for the Dirichlet heat kernel and the lateral Poisson kernel and their derivatives, which might be of their own interest. The key result of this section, Lemma~\ref{lem:f}, gives boundary estimates for functions defined by parabolic sweeping. The results of Section~\ref{sec:estimates} are given for the unit ball, which is enough for our purposes, but the arguments work for all bounded $C^{1,1}$ sets after simple modifications. {We note that the fractional Poisson kernels are obtained by integration of the Green kernels against the L\'evy measure, not by taking normal derivatives as in the classical case. Therefore, the approaches to obtaining bounds differ essentially between the classical and the fractional case. Sharp estimates of the lateral Poisson kernel for second order parabolic operators were given by Riahi \cite{MR2139572}, see also Fabes and Salsa \cite{MR709573} for fundamental results on the classical caloric measure.}

The article is organized as follows. Section~\ref{sec:prelim} contains preliminaries, Section~\ref{sec:estimates} is devoted to estimating the kernels, in Section~\ref{sec:mvpdist} we show that caloricity implies being a distributional solution, in Section~\ref{sec:mvpclass} we show that caloric functions are classical solutions for sufficiently good data, and in Section~\ref{sec:distmvp} we show that the distributional solutions are caloric. 

\section*{Acknowledgments} The author would like to thank the anonymous referee for their valuable remarks. 
\section{Preliminaries}\label{sec:prelim} 
All the considered functions and sets are assumed to be Borel. If $f$ is a function of $x$ and other variables, then by $\nabla_x f$, $D^2_x f$, $\Delta^{\alpha/2}_x f$, etc., we mean that the operator is taken with respect to $x$.
\begin{definition}
	Let $I\subseteq \mR$ be an interval and let $(X,d)$ be a metric space. We say that $f\colon I \to X$ is Dini continuous and denote $f\in C^{\rm Dini}(I,X)$, if there exists a modulus of continuity $\omega$ such that
	\begin{align*}
		d(f(t),f(s)) \leq \omega(|t-s|),\ s,t\in I\quad {\rm and}\quad \int_0^1 \frac{\omega(t)}{t}\, dt < \infty.
	\end{align*}
	{We say that $f\in C^{\rm Dini}_{\rm loc}(I,X)$ if it is Dini continuous in any interval relatively compact in $I$.}
\end{definition} 
\begin{definition}
	{We say that an open set $D$ is Lipschitz if there exist $\lambda>0$ and $r_0>0$, such that for all $Q\in \partial D$, up to a rigid motion, $B(Q,r_0)\cap \partial D$ is a graph of a Lipschitz function with Lipschitz constant $\lambda$.}
\end{definition}
On $C^2_c(\mR^d)$, the fractional Laplacian coincides with the generator of the transition semigroup $P_t$ corresponding to the isotropic $\alpha$-stable process, see Sato \cite{MR1739520}. $P_t$ is a convolution semigroup: $P_t f(x) = p_t\ast f(x)$ and 
the \textit{heat kernel} $p_t$ has the following sharp estimates \cite{MR0119247}:
\begin{align}\label{eq:Blum}
	p_t(x) \approx t^{-d/\alpha} \wedge \frac{t}{|x|^{d+\alpha}},\quad t>0,\ x\in\mR^d.
\end{align}
Using the argument given in the proof of \cite[Lemma~5]{MR2283957} one can show that  for $\eps>0$
\begin{align}\label{eq:derivest}
	|D_x^n p_t(x)| \leq C(d,\eps),\quad t>0, |x|>\eps.
\end{align}
The Dirichlet heat kernel of the set $D$ is defined by Hunt's formula \cite{MR1329992}:
\begin{equation}\label{eq:Hunt}
	\begin{split}
	p_t^D(x,y) &= p_t(x,y) - \mE^x[p_{t-\tau_D}(X_{\tau_D},y);\, \tau_D< t]\\
	&= p_t(x,y) - \mE^y[p_{t-\tau_D}(x,X_{\tau_D});\, \tau_D< t],\quad t>0,\ x,y\in D.
	\end{split}
\end{equation}
In particular, $p_t^D(x,y) = p_t^D(y,x)$.  We have $\partial_t p_t^D(x,y) = \Delta^{\alpha/2}_xp_t^D(x,y)$ pointwise, see, e.g., \cite[Lemma~2.5]{HK}. The Dirichlet semigroup $P_t^D$ is defined for, e.g., $f\geq 0$ as \begin{align*}P_t^D f(x) = \int_D p_t^D(x,y)f(y)\, dy, \quad x\in D.\end{align*}
Recall the approximate factorization for the Dirichlet heat kernel \cite{MR2722789}:
\begin{align}\label{eq:factorization}
	p_t^D(x,y)\approx \mP^x(\tau_D > t)p_t(x,y)\mP^y(\tau_D>t),\quad t\in (0,T),\ x,y\in D.
\end{align}
For $D$ of class $C^{1,1}$ the estimates take on a more explicit form \cite{MR2677618}:
\begin{equation}\label{eq:CKS}
	p_t^D(x,y) \approx \bigg(1\wedge \frac{\delta_D(x)^{\alpha/2}}{\sqrt{t}}\bigg)p_t(x,y)\bigg(1\wedge \frac{\delta_D(y)^{\alpha/2}}{\sqrt{t}}\bigg),\quad x,y \in D,\ t\in (0,T).
\end{equation}
We also have the following large time estimates for bounded $C^{1,1}$ sets $D$ \cite{MR2677618,MR1643611}:
\begin{equation}\label{eq:largetimes}
	p_t^D(x,y) \approx e^{-\lambda_1 t}\delta_D(x)^{\alpha/2}\delta_D(y)^{\alpha/2},\quad x,y\in D,\ t\in (T,\infty).
\end{equation}
Here, $\lambda_1>0$ is the first eigenvalue for the Dirichlet fractional Laplacian in $D$. The integral representation \eqref{eq:urepr} of caloric functions uses the parabolic Martin kernel, defined as
\begin{align*}
	\eta_{t,Q}(x) = \lim\limits_{D\ni y\to Q} \frac{p_t^D(x,y)}{\mP^y(\tau_D>1)},\quad t>0,\ x\in D,\ Q\in\partial D.
\end{align*}
The existence of the parabolic Martin kernel in various settings was obtained in, e.g, \cite{MR4635724,HK, bogdan2018}, see also \cite{MR3462074}. 
Similar to the elliptic case \cite{MR3393247,MR1704245}, the parabolic Martin kernel is associated with singular solutions, {as the last term in the representation \eqref{eq:urepr} appears (i.e., is non-zero) only for caloric functions which are unbounded at the lateral boundary.}

We recall the Ikeda--Watanabe formula from \cite{ikeda1962} (see also \cite[(4.13)]{MR3737628}):
\begin{equation}\label{eq:IW}
	\mP^{x} \big[ \tau_{D} \in I, X_{\tau_{D}-} \in A, X_{\tau_{D}} \in B \big] = \int_{I} \int_{B} \int_{A} \nu(y,z)p_{u}^{D} (x,dy) \, dz\, du,
\end{equation}
where $I \subset (0,\infty)$, $A \subset D$, and $B \subset ( \overline{D} )^{c}$. We define the lateral Poisson kernel as follows:
\begin{align*}
	J^D(t,x,s,z) = \int_D p_{t-s}^D(x,y)\nu(y,z)\, dy,\quad s<t,\ x\in D,\ z\in D^c.
\end{align*}
Note that $J^D(t,x,s,z) = J^D(t-s,x,0,z)$. {To make connection with potential theory, we note that $J^D(t,x,s,z)\textbf{1}_{(0,t)}(s)\, dz\, ds + p_t^D(x,y)\, dy$ can be referred to as the caloric measure of the cylinder $(0,t)\times D$, see \cite[1.XV.7]{MR1814344} for the classical case.} 

{We also recall the Green function and the Poisson kernel of $D$:
	\begin{align}\label{eq:PDGD}
		G_D(x,y) := \int_0^\infty p_t^D(x,y)\, dt,\quad P_D(x,z) := \int_D G_D(x,v)\nu(v,z)\, dv,\quad x,y\in D,\ z\in D^c.
	\end{align}
By \eqref{eq:IW}, $P_D(x,\cdot)$ is a density of a sub-probability measure. If $D$ is, e.g., bounded and Lipschitz we have $\int_{D^c} P_D(x,z)\, dz = 1$.}
\begin{lemma}\label{lem:JDdelta}
	Let $D$ be bounded and Lipschitz. Then, $\lim\nolimits_{t-s \to 0^+} J^D(t,x,s,z) = \nu(x,z)$ holds for all $x\in D$ and $z\in D^c$,
	and for any $g\in L^1({B_1^c}, 1\wedge \nu)$ locally bounded on a neighborhood of $\overline{D}$,
	\begin{align*}
	\lim\limits_{t-s \to 0^+} \int_{D^c} J^D(t,x,s,z) g(z)\, dz = \int_{D^c} \nu(x,z) g(z)\, dz,\quad x\in D.
	\end{align*}
\end{lemma}
\begin{proof}
	Since $\nu(\cdot,z)\in C_b^\infty(\overline{D})$ for fixed $z\in D^c$, the first convergence is clear. Let $g\in L^1({B_1^c}, 1\wedge \nu)$ and let $r$ be small so that $B(x,2r)\subset D$. By Fubini's theorem,
	\begin{align*}
		\int_{D^c} J^D(t,x,s,z) g(z)\, dz = \bigg(\int_{B(x,r)} \int_{D^c} +  \int_{D\setminus B(x,r)}\int_{D^c}\bigg) g(z) p_{t-s}^D(x,y)\nu (y,z)\, dy\, dz.  	\end{align*}
		The first integral converges to $\textstyle  \int_{D^c} \nu(x,z) g(z)\, dz$ by the dominated convergence theorem (note that $\nu(y,z)\approx \nu(x,z)$ on the integration set). The second integral converges to 0 by \cite[Lemma~B.1]{HK} and the dominated convergence theorem.
\end{proof}
\begin{lemma}
	The function $D^c\ni z\mapsto J^D(t,x,s,z)$ is integrable for any $s<t$ and $x\in D$. If $D$ is bounded and Lipschitz, then for fixed $x$ the integrals are uniformly bounded with respect to $s$ and $t$.
\end{lemma}
\begin{proof}
	By \eqref{eq:IW}, for any open $D$, $x\in D$, and $t\in \mR$ we have
	\begin{align*}
		\int_{\-\infty}^t \int_{D^c} J^D(t,x,s,z)\, dz\, ds \leq 1.
	\end{align*}
	In particular, the integrals over $D^c$ are finite for almost all $s$ and the intrinsic ultracontractivity \cite{MR1643611} implies that they are finite for all $s<t$.
	
	The uniform bound for $s\leq t-1$ and arbitrary bounded open $D$ follows from, e.g., \cite[Proposition~4.3]{MR2569321} which implies that
	\begin{align*} \int_{D^c} J^D(t,x,s,z)\, dz \lesssim \int_{D^c} J^D(t,x,t-1,z)\, dz.\end{align*}  If $D$ is bounded and Lipschitz and $x$ is fixed, then the same bound for $s$ close to $t$ is a consequence of \cite[Lemma~5.16]{HK}.
\end{proof}
\begin{lemma}\label{lem:ptdu0}
	Assume that $D$ is bounded and Lipschitz and let $u_0\colon D\to [0,\infty)$ be Borel. If there exist $t_0>0$ and $x_0\in D$ such that $P_{t_0}^D u_0(x_0)$ is finite, then
	\begin{enumerate}
		\item $P_t^Du_0(x)$ is finite for all $t>0$ and $x\in D$,
		\item $\int_D \mP^y(\tau_D>1)u_0(y)\, dy < \infty$, in particular $u_0\in L^1_{\rm loc}(D)$,
		\item $P_t^D u_0\in C_b(D)\cap C^{\infty}(D)$ for all $t>0$,
		\item The map $(t,x) \mapsto P_t^D u_0(x)$ is in $L^1((0,T)\times D)$ for any $T>0$.
	\end{enumerate}
\end{lemma}
\begin{proof}
	For any $t>0$ and $x\in D$ we have $p_t^D(x,y) \approx p_{t_0}^D(x_0,y)$ uniformly in $y$, which proves (1). By \eqref{eq:factorization},
	\begin{align*}
		P_1^D u_0(x) \approx \mP^x(\tau_D>1)\int_D \mP^y(\tau_D>1)u_0(y)\, dy,\quad x\in D,
	\end{align*}
	which proves (2). Furthermore, for any $t>0$,
	\begin{align*}
		\int_D P_t^D u_0(x)\, dx &= \int_D \int_D p_t^D(x,y)u_0(y)\, dy\\
		&= \int_D\mP^y(\tau_D>t) u_0(y)\, dy < \infty,
	\end{align*}
	that is, $P_t^D u_0\in L^1(D)$ and by a bootstrap argument $P_t^D u_0 \in C_b(D)$. Furthermore, using \eqref{eq:Hunt} and \eqref{eq:derivest} we get $P_t^D u_0 \in C_b(D)\cap C^\infty(D)$ for all $t>0$. Finally, by Tonelli's theorem we have
	\begin{align*}
		\int_0^t \int_D P_s^D u_0(x)\, dx\, ds = \int_0^t \int_D \mP^y(\tau_D>t)u_0(y)\, dy.
	\end{align*}
	Therefore, by \cite[Lemma~B.2]{HK} there exists $\sigma\in (0,1)$ such that
	\begin{align*}
		\int_0^t \int_D P_s^D u_0(x)\, dx\, ds\lesssim t^{\sigma}\int_0^ts^{-\sigma} \int_D  \mP^y(\tau_D>t) u_0(y)\, dy\, ds <\infty,
	\end{align*}
	which ends the proof.
\end{proof}
\section{Auxiliary estimates on the kernels}\label{sec:estimates}
In the next result we give off-diagonal upper bounds for the derivatives of the Dirichlet heat kernel, which may be of independent interest. For $n=1$ the bound follows from Kulczycki and Ryznar \cite{MR3767143}. {In the sequel, we denote $B_r = B(0,r)$ for $r>0$.}
\begin{remark}
	In this section we work with $D=B_1$ in order to make the formulations and proofs clear. The arguments can be modified with almost no difficulty to work for bounded $C^{1,1}$ sets. At some point, one has to use the Poisson kernel estimates \cite{Chen1998} instead of the explicit formula.
\end{remark}
\begin{theorem}\label{th:DHKder}
	{For $n=0,1,2,\ldots$},
	\begin{align*}
		|D^n_xp_t^{B_1}(x,y)|\lesssim 1\wedge \frac{\delta_{B_1}(y)^{\alpha/2}}{\sqrt{t}},\quad  t\in(0,T),\ x\in B_{1/4},\  y\in B_1\setminus B_{1/2}.
	\end{align*}
	Furthermore, if $\lambda_1$ is the first Dirichlet eigenvalue for $(-\Delta)^{\alpha/2}$ in $B_1$, then, for $n=0,1,2,\ldots$,
	\begin{align}\label{eq:largeder}
		|D^n_x p_t^{B_1}(x,y)| \lesssim e^{-\lambda_1 t} \delta_{B_1}(y)^{\alpha/2},\quad  t\geq T,\ x\in B_{1/2},\  y\in B_1.
	\end{align}
\end{theorem}
We will prove Theorem~\ref{th:DHKder} using Hunt's formula and the following auxiliary result.
\begin{lemma}\label{lem:f}
	Assume that $f\colon (0,T)\times (\mR^d\setminus B_{1/2})\to \mR$ is bounded and Lipschitz. 
	Then
	\begin{align*}
		|f(t,y) - \mE^y[f(t-\tau_{B_1},X_{\tau_{B_1}});\, \tau_{B_1} < t]| \lesssim 1\wedge \frac{\delta_{B_1}(y)^{\alpha/2}}{\sqrt{t}},\quad t\in (0,T),\ y\in B_1\setminus B_{1/2}. 
	\end{align*}
\end{lemma}
\begin{proof}
	We have 
	\begin{align*}
			|f(t,y) - \mE^y[f(t-\tau_{B_1},X_{\tau_{B_1}});\, \tau_{B_1} < t]| \leq 	\mE^y[|f(t,y)- f(t-\tau_{B_1},X_{\tau_{B_1}})|;\, \tau_{B_1} < t] + |f(t,y)|\mP^y(\tau_{B_1}\geq t).
	\end{align*}
	Since $f$ is bounded, the estimate for the second term follows from bounds for the survival probability, so it remains to investigate the first term. Since $f$ is Lipschitz and bounded,
	\begin{align*}
			\mE^y[|f(t,y)- f(t-\tau_{B_1},X_{\tau_{B_1}})|;\, \tau_{B_1} < t]  &= \int_0^t \int_{B_1^c}\int_{B_1} |f(t,y) - f(t-s,z)| p_{s}^{B_1}(y,v)\nu(v,z)\, dv\, dz\, ds\\
			&\lesssim  \int_0^t \int_{B_1^c}\int_{B_1} (1\wedge |y-z|)  p_{s}^{B_1}(y,v)\nu(v,z)\, dv\, dz\, ds\\
			&+ \int_0^t \int_{B_1^c}\int_{B_1} s\,  p_{s}^{B_1}(y,v)\nu(v,z)\, dv\, dz\, ds =: I_1 + I_2.
	\end{align*}
	Furthermore, by the definition and the explicit formula of the Poisson kernel $P_{B_1}$ of the ball \cite{MR350027},
	\begin{align*}
		I_1 \leq  \int_{B_1^c} (1\wedge |y-z|) \int_{B_1} \int_0^\infty p_s^{B_1}(y,v)\, ds\, \nu(v,z)\, dv\, dz &= \int_{B_1^c} (1\wedge |y-z|) P_{B_1}(y,z)\, dz\\
		&\approx \int_{B_1^c} \frac{\delta_{B_1}^{\alpha/2}(y)}{\delta_{B_1}^{\alpha/2}(z)} |y-z|^{-d}(1\wedge |y-z|) \, dz\\
		&=  \delta_{B_1}^{\alpha/2}(y)\int_{B_1^c}\delta_{B_1}^{-\alpha/2}(z)(|y-z|^{-d}\wedge |y-z|^{-d+1}) \, dz\\
		&\lesssim  \delta_{B_1}^{\alpha/2}(y)\int_{B_1^c}\delta_{B_1}^{-\alpha/2}(z)(|e_d-z|^{-d}\wedge |e_d-z|^{-d+1}) \, dz,
	\end{align*}
	where $e_d = (0,\ldots,0,1)$. We will show that the remaining integral is finite. Note that there is no problem at infinity, as the integrand is of the order $|z|^{-d-\alpha/2}$. Furthermore, if we let $z = (z',z_d)$, where $z'\in \mR^{d-1}$, $z_d \in \mR$, and $K(0,r)$ be {the ball in $\mR^{d-1}$ centered at 0 with radius $r$}, then
	\begin{align*}
		\int_{B_2\setminus B_1}\delta_{B_1}^{-\alpha/2}(z)(|e_d-z|^{-d}\wedge |e_d-z|^{-d+1}) \, dz &\approx 	\int_{B_2\setminus B_1}\delta_{B_1}^{-\alpha/2}(z) |e_d-z|^{-d+1} \, dz\\
		&\approx \int_0^1\int_{K(0,1)} z_d^{-\alpha/2} (|z'|^2 + z_d^2)^{(-d+1)/2} \, dz'\, dz\\
		&\approx \int_0^1 z_d^{-\alpha/2} \int_{K(0,1/z_d)} (1 + |v'|^2)^{-d+1}\, dv'\\
		&\lesssim \int_0^1 z_d^{-\alpha/2}|\log z_d|\, dz_d <\infty.
	\end{align*}
	This proves that $I_1 \lesssim \delta_{B_1}^{\alpha/2}(y)$. We	 now estimate $I_2$. By Tonelli's theorem and \eqref{eq:CKS},
	\begin{align}\label{eq:Idwa}
		I_2 \approx \int_0^t s\int_{B_1} p_s^{B_1}(y,v)\delta^{-\alpha}_{B_1}(v)\, dv\, ds = \int_0^t s\bigg(1 \wedge \frac{\delta^{\alpha/2}_{B_1}(y)}{\sqrt{s}}\bigg)\int_{B_1} p_s(y,v)\bigg(1\wedge \frac{\delta^{\alpha/2}_{B_1}(v)}{\sqrt{s}}\bigg)\delta_{B_1}^{-\alpha}(v)\, dv\, ds.
	\end{align}
	We will now focus on the remaining integral over $B_1$.
	\begin{align}\nonumber
		&\int_{B_1} p_s(y,v)\bigg(1\wedge \frac{\delta^{\alpha/2}_{B_1}(v)}{\sqrt{s}}\bigg)\delta_{B_1}^{-\alpha}(v)\, dv\\ = &\frac 1{\sqrt{s}}\int_{B_1\cap\{\delta_{B_1}(v) \leq s^{1/\alpha}\}} p_s(y,v)\delta^{-\alpha/2}_{B_1}(v)\, dv + \int_{B_1\cap\{\delta_{B_1}(v) > s^{1/\alpha}\}}p_s(y,v)\delta_{B_1}^{-\alpha}(v)\, dv.\label{eq:deltasplit}
	\end{align}
	Note that 
	\begin{align*}
		\int_{B_1\cap\{\delta_{B_1}(v) > s^{1/\alpha}\}}p_s(y,v)\delta_{B_1}^{-\alpha}(v)\, dv \leq \frac 1s	\int_{B_1\cap\{\delta_{B_1}(v) > s^{1/\alpha}\}}p_s(y,v)\, dv \leq \frac 1s.
	\end{align*}
 	For the first integral in \eqref{eq:deltasplit} we use \eqref{eq:Blum}:
 	\begin{align*}
 	\frac 1{\sqrt{s}}\int_{B_1\cap\{\delta_{B_1}(v) \leq s^{1/\alpha}\}} p_s(y,v)\delta^{-\alpha/2}_{B_1}(v)\, dv &\approx 	\frac 1{\sqrt{s}}\int_{B_1\cap\{\delta_{B_1}(v) \leq s^{1/\alpha}\}} \bigg(s^{-d/\alpha}\wedge \frac{s}{|y-v|^{d+\alpha}}\bigg) \delta^{-\alpha/2}_{B_1}(v)\, dv\\
 	&= \sqrt{s} \int_{B_1\cap\{\delta_{B_1}(v) \leq s^{1/\alpha} \leq |y-v|\}} \frac{\delta^{-\alpha/2}_{B_1}(v)}{|y-v|^{d+\alpha}} \, dv\\
 	&+s^{-d/\alpha-1/2}\int_{B_1\cap\{|y-v|\vee \delta_{B_1}(v) \leq s^{1/\alpha}\}}  \delta^{-\alpha/2}_{B_1}(v)\, dv =: I_{21} + I_{22}.
 	\end{align*}
 	If we let $Q=y/|y|$, then
 	\begin{align*}
 		I_{22} \lesssim s^{-d/\alpha-1/2}\int_{B(Q,s^{1/\alpha})} \delta_{B_1}^{-\alpha/2}(v)\, dv \approx s^{-d/\alpha-1/2} s^{(d-1)/\alpha} \int_0^{s^{1/\alpha}} r^{-\alpha/2}\, dr \approx \frac 1s.
 	\end{align*}
 	Furthermore,
 	\begin{align*}
 		I_{21} &\lesssim \sqrt{s}\int_{\delta_{B_1}(v) < s^{1/\alpha} < |Q-v|} |Q-v|^{-d-\alpha}\delta_{B_1}^{-\alpha/2}(v)\, dv\\
 		&\approx \sqrt{s} \int_0^{s^{1/\alpha}}\int_{K(0,1)\setminus K(0,s^{1/\alpha})} |v|^{-d-\alpha} v_d^{-\alpha/2}\, dv' \, dv_d\\
 		&\approx \sqrt{s}\int_0^{s^{1/\alpha}} v_d^{-d-3\alpha/2}\int_{s^{1/\alpha}}^1 \Big(\Big(\frac{r}{v_d}\Big)^2 + 1\Big)^{-\frac{d+\alpha}2} r^{d-1}\, dr\, dv_d\\
 		&\leq \sqrt{s}\int_0^{s^{1/\alpha}} v_d^{-3\alpha/2}\int_{s^{1/\alpha}/v_d}^\infty (1+r^2)^{-\frac{d+\alpha}{2}} r^{d-1}\, dr\, dv_d\\
 		&\approx \frac 1{\sqrt{s}}\int_0^{s^{1/\alpha}} v_d^{-\alpha/2}\, dv_d \approx s^{1/\alpha-1}\lesssim \frac 1s.
 	\end{align*}
 	Coming back to \eqref{eq:deltasplit} and \eqref{eq:Idwa} we get $I_2\lesssim \delta_{B_1}^{\alpha/2}(y)$. This ends the proof.
\end{proof}
\begin{proof}[Proof of Theorem~\ref{th:DHKder}]
	Let $\beta$ be a multi-index with $|\beta| = n$. It suffices to differentiate Hunt's formula and take $f(t,y) = \partial_x^\beta p_t(x,y)$ in Lemma~\ref{lem:f}. Note that in our setting $|x-y|\geq 1/4$, so $f$ indeed satisfies the assumptions of the lemma.
	For the large time estimate, assume without loss of generality that $T=1$. By Hunt's formula and \eqref{eq:derivest} we get the following bound:
	\begin{align}\label{eq:p12}
		|D^n_x p_{1/2}^{B_1}(x,y)| \lesssim |D^n_x p_{1/2}(x,y)| + \mE^y[|D^n_x p_{1/2-\tau_{B_1}}(x,X_{\tau_{B_1}})|;\, \tau_{B_1}<1/2] \lesssim 1,\quad x\in B_{1/2},\ y\in B_1.
	\end{align}
	By this, \eqref{eq:largetimes}, and the Chapman--Kolmogorov equation:
	\begin{align*}
	|D^n_x p_t^{B_1}(x,y)| \leq \int_{B_1} |D^n_x p_{1/2}^{B_1}(x,v)| p_{t-1/2}^{B_1}(v,y)\, dv \lesssim e^{-\lambda_1 t} \delta_{B_1}(y)^{\alpha/2}.
	\end{align*}
	\end{proof}
\begin{theorem}\label{th:JBest}
		For any $s<t$, $x\in B_1$ and $z\in B_1^c$,
		\begin{align*}
			\partial_t J^{B_1}(t,x,s,z) = \Delta^{\alpha/2}_x J^{B_1}(t,x,s,z) = \int_{B_1} \partial_t p_{t-s}^{B_1}(x,y)\nu(y,z)\, dy = \int_{B_1}  \Delta^{\alpha/2}_x p_{t-s}^{B_1}(x,y)\nu(y,z)\, dy.
		\end{align*}
		Furthermore, if $x\in B_{1/2}$, then
		\begin{align}\label{eq:JBest}
			|\partial_t J^{B_1}(t,x,s,z)|\lesssim \begin{cases}\frac 1{t-s} + \frac{\delta_{B_1}(z)^{-\alpha/2}}{(t-s)^{1-(2-\alpha)/2\alpha}},\quad &0<t-s<T,\ z\in B_2\setminus B_1,\\
				\frac{|z|^{-d-\alpha}}{t-s},\quad &0<t-s<T,\ z\in B_2^c.\end{cases}
		\end{align}
		and
		\begin{align}\label{eq:JBlarge}
			|\partial_t J^{B_1}(t,x,s,z)|\lesssim \begin{cases}e^{-\lambda_1 t} \delta_{B_1}(z)^{-\alpha/2},\quad &t-s\geq T,\ z\in B_2\setminus B_1,\\
			e^{-\lambda_1 t} |z|^{-d-\alpha},\quad &t-s\geq T,\ z\in B_2^c.\end{cases}
		\end{align}
	\end{theorem}
	\begin{proof} By Theorem~\ref{th:DHKder}, \eqref{eq:p12}, the fact that  $\Delta^{\alpha/2}_xp_t^{B_1}(x,y)$ is bounded when $t$ is separated from 0 and $p_t^{B_1}$ satisfies the fractional heat equation (see \cite[Corollary~2.6, Lemma~A.1]{HK}), the integrals commute with $\partial_t$ and $\Delta^{\alpha/2}_x$ by the dominated convergence theorem, therefore it suffices to establish \eqref{eq:JBest} and \eqref{eq:JBlarge}.
		
	In order to get \eqref{eq:JBest} we will use one of the following general estimates, depending on $\alpha$: for $0<r<1/2$ and $x\in B_{1/4}$,
	\begin{align}
		&|\Delta^{\alpha/2}u(x)| = \bigg|\int_{\mR^d} (u(y) - u(x) - \nabla u(x)(y-x)\textbf{1}_{B_{{1/2}}(x)}(y))\nu(y-x)\, dy\bigg| \nonumber\\ \label{eq:c2bound}
		&\lesssim \begin{cases} \|u\|_{L^\infty(B_{1/2})} + \int_{B_{1/2}^c} |u(y)|\nu(y)\, dy + (1 + r^{1-\alpha})\|\nabla u\|_{L^\infty(B_{1/2}\setminus B_r)} + \|D^2 u\|_{L^\infty(B_r)} r^{2-\alpha},\quad &\alpha\in (1,2),\\
		\|u\|_{L^\infty(B_{1/2})} r^{-1}+ \int_{B_{1/2}^c} |u(y)|\nu(y)\, dy + \|D^2 u\|_{L^\infty(B_r)} r,\quad &\alpha = 1,\\
		\|u\|_{L^\infty(B_{1/2})} r^{-\alpha} + \int_{B_{1/2}^c} |u(y)|\nu(y)\, dy + \|\nabla u\|_{L^\infty(B_r)} r^{1-\alpha},\quad &\alpha\in (0,1).
		\end{cases}
	\end{align}
	We take $u(x) = J^{B_1}(t,x,0,z)$ and we estimate the terms appearing in \eqref{eq:c2bound}. 
	
	\textbf{Case 1.} $z\in B_2\setminus B_1$. First, by \cite[Lemma~B.1]{HK}, for $x\in B_{1/2}$ we have
	\begin{align*}
		|u(x)| \leq \int_{B_1\setminus B_{3/4}} p_1^{B_1}(x,y)\nu(y,z)\, dy + \int_{B_{3/4}} p_t(x,y)\nu(y,z)\, dy \lesssim \delta_{B_1}(z)^{-\alpha/2}+ |z|^{-d-\alpha}
	\end{align*}
	and
	\begin{align*}
		\int_{B_{1/2}^c} |u(x)|\nu(x)\, dx &= \int_{B_{1/2}^c}\int_{B_1} p_t^{B_1}(x,y)\nu(y,z)\, dy\, \nu(x)\, dx\\ &\lesssim \int_{B_1} \bigg(1\wedge \frac{\delta_{B_1}(y)^{\alpha/2}}{\sqrt{t}}\bigg)\nu(y,z)\int_{B_1\setminus B_{1/2}} p_t(x,y)\, dx\, dy \lesssim  \frac{\delta_{B_1}(z)^{-\alpha/2}}{\sqrt{t}}.
	\end{align*}
	We note that by Theorem~\ref{th:DHKder} and Hunt's formula \eqref{eq:Hunt}, for any $x\in B_{1/2}$ and $t\in (0,1)$,
	\begin{align}\label{eq:DHKnabla}
		|\nabla_x p_t^{B_1}(x,y)| \lesssim \begin{cases}
			1\wedge \frac{\delta_{B_1}(y)^{\alpha/2}}{\sqrt{t}},\quad &y\in B_1\setminus B_{3/4},\\
			1 + |D_x p_t(x,y)|,\quad &y\in B_{3/4}.
		\end{cases}
	\end{align}
	Hence, by, e.g., \cite[(2.6)--(2.9)]{MR1881259} and the fact that $z\in B_2\setminus B_1$, for $x\in B_{1/2}$ we have
	\begin{align*}
		|\nabla u(x)| \lesssim \int_{B_{3/4}} (1 + |\nabla_x p_t(x,y)|)\, dy + \int_{B_1\setminus B_{3/4}} \bigg(1\wedge \frac{\delta_{B_1}(y)^{\alpha/2}}{\sqrt{t}}\bigg)\nu(y,z)\, dy \lesssim t^{-1/\alpha} + \frac{\delta_{B_1}(z)^{-\alpha/2}}{\sqrt{t}}.
	\end{align*}
	A similar argument leads to
		\begin{align*}
		|D^2 u(x)| \lesssim  t^{-2/\alpha} + \frac{\delta_{B_1}(z)^{-\alpha/2}}{\sqrt{t}},\quad x\in B_{1/2}
	\end{align*}
	Taking $r = t^{1/\alpha}$ in \eqref{eq:c2bound}, we obtain the estimate in the case $z\in B_2\setminus B_1$.
	
	\textbf{Case 2.} $z\in B_2^c$. Note that for $y\in B_1$ we have $\nu(y,z)\approx  |z|^{-d-\alpha}$. Thus,
	\begin{align*}
		|u(x)| &\lesssim |z|^{-d-\alpha}\int_{B_1} p_t^{B_1}(x,y)\, dy \approx |z|^{-d-\alpha},\quad x\in B_{1/2}\\
		\int_{B_{1/2}^c} |u(x)|\nu(x)\, dx &\lesssim |z|^{-d-\alpha} \int_{B_1\setminus B_{1/2}}  \int_{B_1} p_t^{B_1}(x,y)\nu(x) \, dy \, dx \approx |z|^{-d-\alpha}.
	\end{align*}
	Proceeding as in the previous case we also get that for $x\in B_{1/2}$,
	\begin{align*}
		|\nabla u(x)| &\lesssim t^{-1/\alpha} |z|^{-d-\alpha},\\
		|D^2u(x)| &\lesssim t^{-2/\alpha}|z|^{-d-\alpha}.
	\end{align*}
	Taking $r=t^{1/\alpha}$ in \eqref{eq:c2bound} we obtain \eqref{eq:JBest}.
	
	Estimate \eqref{eq:JBlarge} is a direct consequence of \eqref{eq:largeder}, which yields
	\begin{align*}
		|D^n_x J^{B_1}(t,s,x,z)| \lesssim \int_{B_1} |D^n_x p_{t-s}^{B_1}(x,y)|\nu(y,z)\, dy \lesssim e^{-\lambda_1 t} \int_{B_1} \delta_{B_1}(y)^{\alpha/2}\nu(y,z)\, dy.
	\end{align*}
	The last integral is bounded by (a multiple of) $\delta_{B_1}(z)^{-\alpha/2}$ for $z\in B_2\setminus B_1$ and $|z|^{-d-\alpha}$ for $|z|\geq 2$. This ends the proof.
	\end{proof}
\section{Caloric functions are distributional solutions}\label{sec:mvpdist}
\begin{theorem}\label{th:mvptodist}
	If $u$ is caloric in $[0,T)\times D$, then it is a distributional solution to \eqref{eq:FHE}.
\end{theorem}
\begin{proof}
	Since for any $U\subset\subset D$ there exists a bounded Lipschitz set $V$ such that $U\subset V\subset\subset D$, we can assume that $D$ is Lipschitz. Then $u$ has the representation \eqref{eq:urepr}. We first show that $P_t^Du_0$ is a distributional solution and then proceed with the remaining parts assuming that $u_0=0$.
	
	We first note that since the expectation in \eqref{eq:para} is finite, by the monotone convergence we have
	\begin{align*}
		P_t^Du_0(x) = \sup\limits_{U\subset\subset D} P_t^U u_0(x) < \infty,\quad x\in D.
	\end{align*}	
	Furthermore, by Lemma~\ref{lem:ptdu0}, $u_0\in L^1_{\rm loc}(D)$. Let $\phi\in C_c^\infty([0,T)\times D)$. We claim that 
	\begin{align}\label{eq:hto0}
	\lim\limits_{h\to 0^+} \int_D \phi(h,x) P_h^D u_0(x)\, dx = \int_D \phi(0,x)u_0(x)\, dx.
	\end{align}
	Indeed, if we fix $U\subset\subset D$ such that $\textstyle \bigcup\nolimits_{t\in [0,T)}\supp\, \phi(t,\cdot)\subset \subset U$, then we get
	\begin{align*}
		&\bigg|\int_D \phi(h,x) P_h^D u_0(x)\, dx - \int_D \phi(0,x) P_h^D u_0(x)\, dx\bigg|\\
		\leq &\int_U |\phi(h,x) - \phi(0,x)| P_h^D u_0(x)\, dx
		\leq h \int_{U} P_h^D u_0(x)\, dx,
	\end{align*}
	which by Lemma~\ref{lem:ptdu0} and \cite[Lemma~B.2]{HK} converges to 0 as $h\to 0^+$. Therefore, it suffices to show that
	\begin{align}\label{eq:Usemigroup}
		\lim\limits_{h\to 0^+}\int_U |P_h^D u_0(x) - u_0(x)|\, dx = 0.
	\end{align}
	We estimate as follows:
	\begin{align*}
		\int_U |P_h^D u_0(x) - u_0(x)|\, dx \leq \int_U |P_h^D [u_0\cdot\textbf{1}_U](x) - u_0(x)|\, dx + \int_U P_h^D [u_0\cdot\textbf{1}_{U^c}](x)\, dx.
	\end{align*}
	The first integral converges to 0 as $h\to 0^+$ because $u_0\cdot \textbf{1}_U\in L^1(D)$. For the second integral consider an open set $U'$ such that $U\subset\subset U'\subset\subset D$. By \cite[Lemma~B.1]{HK} there exists $\sigma\in (0,1)$ such that $p_h^D(x,y) \lesssim h^{1-\sigma} p_1^D(x,y)$ for $x\in U$ and $y\in D\setminus U'$. Therefore,
	\begin{align*}
		\int_U P_h^D [u_0\cdot\textbf{1}_{U^c}](x)\, dx &= \int_U \int_{D\setminus U'} p_h^D(x,y)\, u_0(y)\, dy\, dx + \int_U \int_{U'\setminus U} p_h^D(x,y)\, u_0(y)\, dy\, dx\\
		&\lesssim h^{1-\sigma} \int_{U}\int_{D\setminus U'} p_1^D(x,y)u_0(y)\, dy + \int_{U}\int_{U'\setminus U} p_h^D(x,y)u_0(y)\, dy. 
	\end{align*}
	Both terms converge to $0$ as $h\to 0^+$; in the second we use the dominated convergence theorem and the fact that $u_0\in L^1(U')$. This proves \eqref{eq:Usemigroup} and thus \eqref{eq:hto0} follows.
	
	By \eqref{eq:hto0} we get
	\begin{align*}
		\int_{D} \phi(t,x) P_t^D u_0(x)\, dx - \int_D \phi(0,x)u_0(x)\, dx = \lim\limits_{h\to 0^+} \bigg(\int_{D} \phi(t,x) P_t^D u_0(x)\, dx - \int_{D} \phi(h,x) P_h^Du_0(x)\, dx\bigg).
	\end{align*}
	We have
	\begin{align*}&\int_{D} \phi(t,x) P_t^D u_0(x)\, dx - \int_{D} \phi(h,x) P_h^Du_0(x)\, dx\\ =\, &\int_D \int_h^t \partial_\tau (\phi(\tau,x) P_\tau^D u_0(x))\, d\tau\, dx\\
		=\, &\int_h^t \int_D \partial_\tau \phi(\tau,x) P_\tau^D u_0(x)\, dx\, d\tau + \int_h^t \int_D \phi(\tau,x) \partial_\tau P_\tau^D u_0(x)\, dx\, d\tau.
	\end{align*}
	The equalities above make sense because $P_\tau^D u_0\in L^2(D)$ for all $\tau>0$, so by, e.g., \cite[Section~2.3]{HK} we have $\partial_\tau P_\tau^D u_0 = \partial_\tau P_{\tau/2}^D P_{\tau/2}^D u_0 = \Delta^{\alpha/2}P_\tau^D u_0.$  Furthermore, since $P_{\tau/2}^Du_0\in L^2(D)$ and $\phi(\tau,\cdot)\in C_c^\infty(D)$, we get
	\begin{align*}
	&\int_h^t \int_D \partial_\tau \phi(\tau,x) P_\tau^D u_0(x)\, dx\, d\tau + \int_h^t \int_D \phi(\tau,x) \partial_\tau P_\tau^D u_0(x)\, dx\, d\tau\\
	=\, &\int_h^t \int_D \partial_\tau \phi(\tau,x) P_\tau^D u_0(x)\, dx\, d\tau + \int_h^t \int_D \Delta^{\alpha/2}\phi(\tau,x) P_\tau^D u_0(x)\, dx\, d\tau.
	\end{align*}
	By the fact that $(\partial_\tau  + \Delta^{\alpha/2})\phi$ is bounded, by Lemma~\ref{lem:ptdu0}, and the dominated convergence theorem we obtain
	\begin{align*}
		\int_D\phi(t,x) P_t^D u_0(x)\, dx - \int_D \phi(0,x) u_0(x)\, dx = \int_0^t\int_D (\partial_{\tau} +\Delta^{\alpha/2})\phi(\tau,x) P_\tau^D u_0(x)\, dx.
	\end{align*}
	Therefore, in the remainder of the proof we can and do assume that $u_0 \equiv 0$, i.e., we will show \eqref{eq:weakform} for caloric $u$ which satisfies
	\begin{align*}
		u(t,x) = \int_0^t \int_{U^c} J^U(t,x,s,z) u(s,z)\, dz\, ds,\quad U\subset\subset D,\ x\in U,\ t\in (0,T).
	\end{align*}
	For $u$ as above, by Fubini--Tonelli (see also \cite[Proposition~5.15]{HK}) we have
	\begin{align*}
		\int_D \phi(t,x) u(t,x)\, dx &= \int_D \int_0^t \int_{U^c}\int_U \phi(t,x) p_{t-s}^U(x,y)\nu(y,z) u(s,z)\, dy\, dz\, ds\, dx\\
		&=\int_0^t \int_{U^c}\int_U P_{t-s}^U \phi(t)(y)\nu(y,z) u(s,z)\, dy\, dz\, ds\\
		&= \int_0^t \int_{U^c}\int_U (P_{t-s}^U \phi(t)(y) - \phi(s,y)) \nu(y,z) u(s,z)\, dy\, dz\, ds\\
		&+ \int_0^t \int_{U^c}\int_U \phi(s,y) \nu(y,z) u(s,z)\, dy\, dz\, ds =: I_1 + I_2.
	\end{align*}
	By the regularity of $\phi$, we find that
	\begin{align*}
		I_1 &=  \int_0^t \int_{U^c}\int_U \int_s^t\partial_{\tau} P_{\tau -s}^U\phi(\tau)(y) \nu(y,z) u(s,z)\, d\tau \, dy\, dz\, ds\\
		&=  \int_0^t \int_{U^c}\int_U \int_s^t\partial_{\tau} \bigg(\int_U p_{\tau-s}^U(x,y)\phi(\tau,x)\, dx\bigg) \nu(y,z) u(s,z)\, d\tau \, dy\, dz\, ds\\
		&= \int_0^t \int_{U^c}\int_U \int_s^t \int_U \partial_{\tau}p_{\tau-s}^U(x,y)\phi(\tau,x)\, dx\, \nu(y,z) u(s,z)\, d\tau \, dy\, dz\, ds\\
		&+\int_0^t \int_{U^c}\int_U \int_s^t \int_U p_{\tau-s}^U(x,y)\partial_{\tau}\phi(\tau,x)\, dx\, \nu(y,z) u(s,z)\, d\tau \, dy\, dz\, ds.
	\end{align*}
	Since $\partial_\tau p_\tau^D = \Delta^{\alpha/2}p_{\tau}^D$ and $\phi(\tau,\cdot)\in C_c^\infty(U)$, by Fubini--Tonelli we obtain
	\begin{align*}
		I_1 &= \int_0^t\int_U (\partial_\tau + \Delta^{\alpha/2})\phi(\tau,x)\int_0^\tau \int_{U^c}\int_U p_{\tau-s}^U(x,y)\nu(y,z)u(s,z)\, dy\, dz\, ds\, dx\, d\tau\\
		&= \int_0^t\int_U (\partial_\tau +\Delta^{\alpha/2})\phi(\tau,x)u(\tau,x)\, dx\, d\tau.
	\end{align*}
	Furthermore,
	\begin{align*}
		I_2 = \int_0^t \int_{U^c} u(s,z) \Delta^{\alpha/2}\phi(s,z)\, dz\, ds,
	\end{align*}
	therefore,  noting that $\partial_\tau \phi(\tau,\cdot) = 0$ on $U^c$, we find that
	\begin{align*}
		I_1 + I_2 = \int_0^t \int_{\mR^d} (\partial_\tau +\Delta^{\alpha/2})\phi(\tau,x)u(\tau,x)\, dx\, d\tau.
	\end{align*}
	This ends the proof.
\end{proof}
\section{Caloric functions are classical solutions}\label{sec:mvpclass}
In this section we discuss sufficient conditions on the exterior data $g$ and boundary data $\mu$, under which a caloric function $u$ is regular enough to solve the fractional heat equation pointwise. {In the whole section we assume that the conditions of Theorem~\ref{th:mvpclass} hold. We first give an integrability estimate for the caloric function.}
\begin{lemma}\label{lem:mutog}Assume that $D$ is bounded and Lipschitz, let {$g\in C([0,T],L^1(D^c,P_D(x_0,\cdot))$ for some $x_0\in D$,} $\mu\in C([0,T],\mathcal{M}(\partial D))$, and let $\omega_\mu$ {and $\omega_g$} be moduli of continuity for $\mu\colon [0,T]\to \mathcal{M}(\partial D)$ {and $g\colon [0,T] \to L^1(D^c,P_D(x_0,\cdot))$ respectively}. If 
	\begin{align}\label{eq:uwith0initial}
		u(t,x) = {\int_0^t\int_{D^c} J^D(t,x,\tau,z)g(\tau,z)\, dz\, d\tau  }+\int_0^t\int_{\partial D} \eta_{t-\tau,Q}(x)\, \mu(\tau,dQ)\, d\tau,\quad t\in[0,T],\ x\in D,\end{align} and $u=0$ elsewhere, then $u\in C([0,T],L^1({1\wedge \nu}))$ and there exists $\sigma \in (0,1)$ such that for $0\leq s,t\leq T$,
	\begin{align*}
	\|u(s) - u(t)\|_{L^1(D)} \lesssim  |t-s|^{1-\sigma} (\sup\limits_{t\in [0,T]} \|\mu(t)\|_{TV} + {\sup\limits_{t\in [0,T]} \|g(t)\|_{L^1(P_D(x_0,\cdot))}})+ \omega_\mu(t-s) + {\omega_g(t-s)}.
	\end{align*}
\end{lemma}\begin{proof}
	{We first establish two auxiliary estimates. First,	by \cite[Corollary~3.6]{HK} there exists $\sigma\in(0,1)$ such that
		\begin{align}\label{eq:etaintegralestimate}
			\int_D \eta_{t,Q}(x)\, dx \lesssim t^{-\sigma} \int_D \mP^x(\tau_D>t)p_t(x,Q)\, dx \lesssim t^{-\sigma}.
		\end{align}
		Furthermore, we have
		\begin{align*}
			\int_D J^D(t,x,0,z)\, dx = \int_D \int_D p_t^D(x,y)\nu(y,z)\, dy\, dx = \int_D \mP^y(\tau_D > t)\nu(y,z)\, dy.
		\end{align*}
		Therefore, by  \cite[Lemmas~B.2 and 3.3]{HK}, for any fixed $x_0\in D$ we have
		\begin{align}\label{eq:Jintegralestimate}
			\int_D J^D(t,x,0,z)\, dx \lesssim t^{-\sigma}\int_D \mP^y(\tau_D>1)\nu(y,z)\, dy \lesssim t^{-\sigma} \int_D G_D(x_0,y)\nu(y,z)\, dy &= t^{-\sigma}P_D(x_0,z).
		\end{align}
		Denote the first integral on the right-hand side of \eqref{eq:uwith0initial} by $u_1$. Let $0 \leq s<t\leq T$. Then,
		\begin{align*}
			\|u_1(t) - u_1(s)\|_{L^1(D)} \leq&\ \int_0^s\int_{D^c}|g(t-\tau,z) - g(s-\tau,z)| \int_D J^D(\tau,x,0,z)\, dx\, dz\, d\tau\\
			&+\int_s^t\int_{D^c}g(t-\tau,z)\int_D J^D(\tau,x,0,z)\, dx\, dz\, d\tau\\
			 \leq&\ \int_0^s\tau^{-\sigma }\int_{D^c}|g(t-\tau,z) - g(s-\tau,z)| P_D(x_0,z)\, dz\, d\tau\\
			 &+\int_s^t\tau^{-\sigma}\int_{D^c}g(t-\tau,z) P_D(x_0,z)\, dz\, d\tau\\
			 &\lesssim \omega_g(t-s) + |t-s|^{1-\sigma}\sup\limits_{t\in [0,T]} \|g(t)\|_{L^1(P_D(x_0,\cdot))}.
		\end{align*}
	The estimate for the second term in \eqref{eq:uwith0initial} follows from a similar calculation using \eqref{eq:etaintegralestimate} instead of \eqref{eq:Jintegralestimate}.
	}
\end{proof}
\begin{remark}\label{rem:simp}
{Before we prove interior regularity in space and time we will make some reductions under the assumptions of Theorem~\ref{th:mvpclass}. First, note that by \cite[Proposition~5.15]{HK} $u$ is continuous in $(0,T)\times D$, so by \cite[Theorem~4.14]{MR2008600} it is also locally H\"older continuous therein. Furthermore, by Lemma~\ref{lem:ptdu0}, the fact that $P_t^Du_0 = P_{t-\epsilon}^DP_\epsilon^Du_0$ and the boundedness of $\Delta^{\alpha/2}_xp_t^D(x,\cdot)=\partial_t p_t^D(x,\cdot)$ for fixed $t>0$ (see \cite[Corollary~3.6]{HK}), we find that $P_t^Du_0$ is a classical solution, so we can assume that $u_0=0$. By this, the mean value property on $(\epsilon,T-\epsilon)\times D$, and Lemma~\ref{lem:mutog}, we have $u\in C^{\rm Dini}_{\rm loc}((0,T),L^1(1\wedge \nu))$.  Then, by considering the mean value property on $(\epsilon,T-\epsilon)\times B$ for some $\epsilon>0$ and a ball $B\subset\subset D$, we may assume (after translating and rescaling) that $D=B_1$, $g$ is uniformly H\"older continuous on $(0,T)\times (B_2\setminus B_1)$, and $g\in C^{\rm Dini}((0,T), L^1(B_1^c,1\wedge\nu))$.}

{To summarize, in the remainder of this section we assume without loss of generality that $D=B_1$ and 
\begin{align}\label{eq:uform}
	u(t,x) = \begin{cases}\int_0^t\int_{B_1^c} J^{B_1}(t,x,s,z)g(s,z)\, dz\, ds,\quad &t\in(0,T], \ x\in B_1,\\
		g(t,x),\quad &t\in[0,T],\ x\in B_1^c,\\
	0,\quad &\text{otherwise},\end{cases}
\end{align}
$g$ is uniformly H\"older continuous on $(0,T)\times (B_2\setminus B_1)$ and $g\in C^{\rm Dini}((0,T),L^1(B_1^c,1\wedge \nu))$.}
\end{remark}
\begin{lemma}
	For {$u$ defined in \eqref{eq:uform}}, we have $u(t,\cdot)\in C^{\infty}(B_1)$ for all $t\in(0,T)$.
\end{lemma}
\begin{proof}
We will prove smoothness in $B_{1/2}$, the modification for the whole of $B_1$ is rather obvious. We first note that by the definition of $J^{B_1}$ and Tonelli's theorem,
\begin{align*}
	u(t,x) = \int_0^t P_{t-s}^{B_1}[f(s)](x)\, ds,
\end{align*}
where 
\begin{align*}
	f(s,y) = \begin{cases}\int_{B_1^c} g(s,z)\nu(y,z)\, dz,\quad &s\in (0,T),\ y\in B_1,\\0,\quad &\text{otherwise}.\end{cases}
\end{align*}
By properties of $g$, for each $s$ the function $f(s,\cdot)$ is smooth and we have the following bound:
\begin{align}\label{eq:Dnbound}
	|D^n_y f(s,y)| \lesssim \delta_{B_1}^{-\alpha-n}(y),\quad s\in (0,T),\ y\in B_1.
\end{align}
Consider a cut-off function $\eta\in C_c^\infty(\mR^d)$ satisfying $0\leq\eta\leq 1$, $\eta \equiv 0$ on $B_{3/4}^c$, $\eta \equiv 1$ on $B_{1/2}$, and let $f_1 = \eta f$, $f_2 = (1-\eta)f$, $u_1(t) =\textstyle \int_0^t P_{t-s}^{B_1} [f_1(s)]\, ds$, $u_2(t) =\textstyle \int_0^t P_{t-s}^{B_1} [f_2(s)]\, ds$.

We recall Hunt's formula:
\begin{align*}
	p_t^{B_1}(x,y) = p_t(x,y) - \mE^y [p_{t-\tau_{B_1}}(x,X_{\tau_{B_1}});\, \tau_{B_1} < t],
\end{align*}
which yields
\begin{align*}
	u_1(t,x) = \int_0^t P_{t-s}[f_1(s)]\, ds - \int_0^t \int_{B_1} \mE^y[p_{s-\tau_{B_1}}(x,X_{\tau_{B_1}});\, \tau_{B_1} < t] f_1(s,y)\, dy\, ds.
\end{align*}
Since $f_1(s,\cdot)$ is smooth and compactly supported for all $s$, by \eqref{eq:Dnbound} the first term is smooth. For the second term, note that if $x$ is fixed, the expression $p_{s-\tau_{B_1}}(x,X_{\tau_{B_1}})$ and its derivatives are uniformly bounded, because $X_{\tau_{B_1}}\in B_1^c$ almost surely. Therefore, by the dominated convergence theorem, 
\begin{align*}
	D^n_x \int_0^t \int_{B_1} \mE^y[p_{s-\tau_{B_1}}(x,X_{\tau_{B_1}});\, \tau_{B_1} < t] f_1(s,y)\, dy\, ds= \int_0^t \int_{B_1} \mE^y[D^n_xp_{s-\tau_{B_1}}(x,X_{\tau_{B_1}});\, \tau_{B_1} < t] f_1(s,y)\, dy\, ds,
\end{align*}
which proves that $u_1(t,\cdot)$ is smooth for all $t\in(0,T)$.

For $u_2$, note that
\begin{align}\label{eq:u2}
	u_2(t,x) = \int_0^t P_{t-s}^{B_1}[f_2(s)]\, ds  = \int_0^t \int_{B_1\setminus B_{1/2}} p_{t-s}^{B_1}(x,y) f_2(s,y)\, dy\, ds.
\end{align}
Therefore, the smoothness of $u_2$ follows from Theorem~\ref{th:DHKder}, \eqref{eq:Dnbound} for $n=0$, and the dominated convergence theorem.
\end{proof}
We now proceed to time regularity.
\begin{lemma}\label{lem:fixedt}
	Assume that $g\in L^\infty(B_2\setminus B_1)\cap L^1({B_1^c},1\wedge \nu)$. Then, there exists $C = C(d,\alpha,\delta_{B_1}(x))$ such that 
	\begin{align*}
		\int_{B_1^c} J^{B_1}(t,x,s,z) g(z)\, dz\leq C(\|g\|_{L^\infty(B_2\setminus B_1)} + \|g\|_{L^1(1\wedge \nu)}),\quad 0<s<t.
	\end{align*}
\end{lemma}
\begin{proof}
	Since
	\begin{align*}
		\int_{B_1^c} J^{B_1}(t,x,s,z) g(z)\, dz &= \int_{B_1} p_{t-s}^{B_1}(x,y)\int_{B_1^c} g(z)\nu(y,z)\, dz\, dy\\ &\leq c(\|g\|_{L^\infty(B_2\setminus B_1)} + \|g\|_{L^1(1\wedge \nu)})\int_{B_1} p_{t-s}^{B_1}(x,y) \delta_{B_1}(y)^{-\alpha}\, dy,
	\end{align*}
	the result follows from \eqref{eq:CKS} for $t-s \in (0,1]$, and from \eqref{eq:largetimes} for $t-s>1$. 
\end{proof}
\begin{lemma}\label{lem:tplush}
	Assume that $g\in C((0,T)\times B_2)\cap C((0,T), L^1({B_1^c}, 1\wedge \nu))$. Then, 
	\begin{align*}
		\lim\limits_{h\to 0^+}\frac 1h \int_t^{t+h} \int_{B_1^c} J^{B_1}(t+h,x,s,z) g(s,z)\, dz\, ds = \int_{B_1^c} \nu(x,z)g(t,z)\, dz,\quad t\in (0,T),\ x\in B_1.
	\end{align*}
\end{lemma}
\begin{proof}
	We first note that by Lemma~\ref{lem:fixedt}, \cite[Lemma~5.16]{HK}, and the dominated convergence theorem the function $s\mapsto \int_{B^c} J^{B_1}(t+h,x,s,z) g(s,z)\, dz$ is continuous on $(0,t+h)$. Therefore, by the mean value theorem, there exists $\xi_h\in (t,t+h)$ such that
	\begin{align*}
		\frac 1h \int_t^{t+h} \int_{B_1^c} J^{B_1}(t+h,x,s,z) g(s,z)\, dz\, ds = \int_{B_1^c} J^{B_1}(t+h,x,\xi_h,z) g(\xi_h,z)\, dz.
	\end{align*}
	Furthermore, 
	\begin{align*}
	&\int_{B_1^c}	J^{B_1}(t+h,x,\xi_h,z) g(\xi_h,z)\, dz\\ = &\int_{B_1^c} J^{B_1}(t+h,x,\xi_h,z) (g(\xi_h,z) -g(t,z))\, dz  + \int_{B_1^c}J^{B_1}(t+h,x,\xi_h,z) g(t,z)\, dz.
	\end{align*}
	The first term vanishes as $h\to 0^+$ because of the assumptions on $g$ and the second term converges to $\textstyle\int_{B_1^c}\nu(x,z)g(t,z)\, dz$ as $h\to 0^+$ by Lemma~\ref{lem:JDdelta}. This ends the proof.
\end{proof}
\begin{proof}[Proof of Theorem~\ref{th:mvpclass}]
	Let us first recall Remark~\ref{rem:simp}, {by which we can assume that $D=B_1$ and $u$ is of the form \eqref{eq:uform} with $g$ uniformly H\"older continuous on $(0,T)\times (B_2\setminus B_1)$ in both variables and $g\in C^{\rm Dini}((0,T),L^1(B_1^c,1\wedge \nu))$.} 
	We {now} compute the time derivative of $u$. The difference quotient is
	\begin{align*}
		&\frac 1h \bigg(\int_0^{t+h}\int_{B_1^c} g(s,z) J^{B_1}(t+h,x,s,z) \, dz\, ds - \int_0^{t}\int_{B_1^c} g(s,z) J^{B_1}(t,x,s,z) \, dz\, ds\bigg)\\
		= &\frac 1h\bigg(\int_0^{t}\int_{B_1^c} g(s,z) (J^{B_1}(t+h,x,s,z) - J^{B_1}(t,x,s,z)) \, dz\, ds + \int_t^{t+h}\int_{B_1^c} g(s,z) J^{B_1}(t+h,x,s,z) \, dz\, ds \bigg)\\
		= & \frac 1h\bigg(\int_0^{t}\int_{B_1^c} (g(s,z) - g(t,z)) (J^{B_1}(t+h,x,s,z) - J^{B_1}(t,x,s,z)) \, dz\, ds \\
		+ &  \frac 1h\int_0^{t}\int_{B_1^c} g(t,z) (J^{B_1}(t+h,x,s,z) - J^{B_1}(t,x,s,z)) \, dz\, ds+ \frac 1h\int_t^{t+h}\int_{B_1^c} g(s,z) J^{B_1}(t+h,x,s,z) \, dz\, ds\\
		=: & I_1(h) + I_2(h) + I_3(h).
	\end{align*}
	Note that since $g(t,z)$ does not depend on $s$ we have
	\begin{align*}
		I_2(h) = \frac 1h \bigg(-\int_t^{t+h} + \int_0^h\bigg) \int_{B_1^c} J^{B_1}(t+h,x,s,z) g(t,z)\, dz\, ds. 
	\end{align*}
	Therefore, by Lemma~\ref{lem:tplush} we find that
	\begin{align*}
		\lim\limits_{h\to 0^+} (I_2(h) + I_3(h)) = \int_{B_1^c} J^{B_1}(t,x,0,z)g(t,z)\, dz.
	\end{align*}
	In $I_1$ we first use the fundamental theorem of calculus:
	\begin{align*}
		I_1(h) = \int_0^t\int_{B_1^c}\int_0^1 \partial_t J^{B_1}(t+\lambda h,x,s,z) (g(s,z) - g(t,z))\, d\lambda \, dz\, ds.
	\end{align*}
	With the aim to use the dominated convergence theorem we bound the integrand using \eqref{eq:JBest}:
	\begin{align*}
		&|\partial_t J^{B_1}(t+\lambda h,x,s,z) (g(s,z) - g(t,z))|\\
		 \lesssim &\, |g(s,z) - g(t,z)| \begin{cases}\frac 1{t-s} + \frac{\delta_{B_1}(z)^{-\alpha/2}}{(t-s)^{1-(2-\alpha)/2\alpha}},\quad &0<t-s<T,\ z\in B_2\setminus B_1,\\
				\frac{|z|^{-d-\alpha}}{t-s},\quad &0<t-s<T,\ z\in B_2^c.\end{cases}
		\end{align*}
		Note that the bound is an integrable function. Indeed, for $z\in B_2\setminus B_1$ we use the fact that $g$ is H\"older {in time} on $(0,T)\times B_2\setminus B_1$ and for $z\in B_2^c$ we apply the assumption on Dini continuity of $g$. It follows that
		\begin{align}\label{eq:dtu}
			\partial_t u(t,x) = \int_0^t\int_{B_1^c} \partial_t J^{B_1}(t,x,s,z) (g(s,z) - g(t,z)) \, dz\, ds + \int_{B_1^c} J^{B_1}(t,x,0,z)g(t,z)\, dz.
		\end{align}
		The derivative is also continuous in $t$ and $x$, but for clarity we defer the proof of that to the appendix.
		
		We will now compute $\Delta^{\alpha/2} u(t,x)$ and show that it equals $\partial_t u(t,x)$. For $\eps>0$ let
		\begin{align*}
			\Delta^{\alpha/2}_\eps u(t,x) &= \int_{B_1\setminus B_\eps(x)} (u(t,y) - u(t,x) - \nabla_x u(t,x)(y-x)\textbf{1}_{B_{1/2}(x)}(y){\textbf{1}_{[1,2)}(\alpha)})\nu(y-x)\, dy\\ 
			&+ \int_{B_1^c} (g(t,y) - u(t,x))\nu(y-x)\, dy =: \I(\eps) + \II.
		\end{align*}
		Since $u(t,\cdot)$ is smooth in $B_1$ and $g\in L^1({B_1^c}, 1\wedge \nu)$, we have $\Delta_\eps^{\alpha/2}u(t,x) \to \Delta^{\alpha/2}u(t,x)$ as $\eps\to 0^+$. Furthermore,
		\begin{align*}
			\I(\eps) &= \int_{B_1\setminus B_\eps(x)} \int_0^t \int_{B_1^c} (J^{B_1}(t,y,s,z) - J^{B_1}(t,x,s,z) - \nabla_x J^{B_1}(t,x,s,z)(y-x)\textbf{1}_{B_{1/2}(x)}(y){\textbf{1}_{[1,2)}(\alpha)})\times\\
			&\hspace{40pt} \times (g(s,z) - g(t,z))\, dz\, ds \nu(y-x)\, dy\\
			&+  \int_{B_1\setminus B_\eps(x)} \int_0^t \int_{B_1^c} (J^{B_1}(t,y,s,z) - J^{B_1}(t,x,s,z) - \nabla_x J^{B_1}(t,x,s,z)(y-x)\textbf{1}_{B_{1/2}(x)}(y){\textbf{1}_{[1,2)}(\alpha)})\times\\
			&\hspace{40pt} \times  g(t,z)\, dz\, ds \nu(y-x)\, dy\\
			 &=: \I.1(\eps) + \I.2 (\eps).
		\end{align*}
			If we adopt the convention that $J^{B_1}(t,x,s,z) = 0$ for $x\in B_1^c$, then for $x\in B_{1/2}$,
		\begin{align}\nonumber
			\Delta^{\alpha/2}_x J^{B_1}(t,x,s,z) &= \int_{B_1} (J^{B_1}(t,y,s,z) - J^{B_1}(t,x,s,z) - \nabla_x J^{B_1}(t,x,s,z)(y-x)\textbf{1}_{B_{1/2}(x)}(y){\textbf{1}_{[1,2)}(\alpha)})\times\\
			&\qquad\qquad\times \nu(y-x)\, dy - J^{B_1}(t,x,s,z)\nu(x,B_1^c).\label{eq:fracJB}
		\end{align}
		Therefore, by Dini continuity of $g$, Fubini--Tonelli, the dominated convergence theorem, the proof of Theorem~\ref{th:JBest}, and the fact that $(\partial_t - \Delta^{\alpha/2})J^{B_1}= 0$, we get that $I_1(\eps)\to I_1(0)$ as $\eps\to 0^+$ and
		\begin{align*}
			\I.1(0) &= \int_0^t\int_{B_1^c} \Delta^{\alpha/2}_xJ^{B_1}(t,x,s,z) (g(s,z) - g(t,z))\, dz \, ds\\
			 &+\nu(x,B_1^c) \int_0^t \int_{B_1^c} J^{B_1}(t,x,s,z)(g(s,z) - g(t,z))\, dz\, ds\\
			&= \int_0^t\int_{B_1^c} \partial_t J^{B_1}(t,x,s,z) (g(s,z) - g(t,z))\, dz \, ds\\ &+\nu(x,B_1^c) \bigg(u(t,x) - \int_0^t \int_{B_1^c} J^{B_1}(t,x,s,z) g(t,z)\, dz\, ds\bigg)\\
			&=: \I.1.1 + \I.1.2.
		\end{align*}
		Thus, it suffices to show that $\II + \I.2(\eps) + \I.1.2 \to \int_{B_1^c} J^{B_1}(t,x,0,z)g(t,z)\, dz$ as $\eps\to 0^+$. We have
		\begin{align}\label{eq:eqn}
			\II +  \I.1.2 =  \int_{B_1^c} g(t,y) \nu(y-x)\, dy - \nu(x,B_1^c) \int_0^t \int_{B_1^c} J^{B_1}(t,x,s,z) g(t,z)\, dz\, ds.
		\end{align}
		In order to handle $\I.2(\eps)$ we {note that the Poisson kernel satisfies the formula} 
		\begin{align*}
			P_{B_1}(x,z) = \int_{-\infty}^t J^{B_1}(t,x,s,z)\, ds,\quad x\in B_1,\ z\in B_1^c,
		\end{align*}
		and the function 
		\begin{align*}
		P_{B_1}[g(t)](x) := \begin{cases}
			\int_{B_1^c} g(t,z) P_{B_1}(x,z)\, dz,\quad &x\in B_1,\\
			g(t,x),\quad x\in B_1^c,
		\end{cases}
		\end{align*}
		is smooth and harmonic in $B_1$, see \cite{MR1671973}. Thus,
		\begin{align*}
			\I.2(\eps) &= \int_{B_1\setminus B_\eps(x)} (P_{B_1}[g(t)](y) - P_{B_1}[g(t)](x) - \nabla P_{B_1}[g(t)](x)(y-x)\textbf{1}_{B_{1/2}(x)}(y){\textbf{1}_{[1,2)}(\alpha)})	\nu(x-y)\, dy\\
							&-  \int_{B_1\setminus B_\eps(x)} \int_{-\infty}^0 \int_{B_1^c} (J^{B_1}(t,y,s,z) - J^{B_1}(t,x,s,z) - \nabla_x J^{B_1}(t,x,s,z)(y-x)\textbf{1}_{B_{1/2}(x)}(y){\textbf{1}_{[1,2)}(\alpha)})\times\\
							&\hspace{40pt} \times g(t,z)\, dz\, ds \nu(y-x)\, dy.
			\end{align*}
			By \eqref{eq:JBlarge}, Fubini--Tonelli, the dominated convergence theorem, \eqref{eq:fracJB}, $(\partial_t - \Delta^{\alpha/2})J^{B_1} = 0$, and the harmonicity of $P_{B_1}[g(t)]$, we find that $I_2(\eps) \to I_2(0)$ and
			\begin{align*}
				I_2(0) &= \nu(x,B_1^c)\int_{B_1^c}P_{B_1}(x,z)g(t,z)\, dz - \int_{B_1^c} g(t,y)\nu(y-x)\, dy\\
								&- \int_{-\infty}^0 \int_{B_1^c} \partial_t J^{B_1}(t,x,s,z)g(t,z)\, dz\, ds -\nu(x,B_1^c) \int_{-\infty}^0 \int_{B_1^c} J^{B_1}(t,x,s,z) g(t,z)\, dz\, ds\\
							&= \nu(x,B_1^c) \int_0^t \int_{B_1^c} J^{B_1}(t,x,s,z) g(t,z)\, dz\, ds + \int_{B_1^c} J^{B_1}(t,x,0,z)g(t,z)\, dz - \int_{B_1^c} g(t,y)\nu(y-x)\, dy.
			\end{align*}
			Coming back to \eqref{eq:eqn} and using the definition of $P_{B_1}$ we get
			\begin{align*}
				\II + \I.2(0) + \I.1.2 = \int_{B_1^c} J^{B_1}(t,x,0,z)g(t,z)\, dz,
			\end{align*}
			which implies that $\partial_t u(t,x) = \Delta^{\alpha/2} u(t,x)$. This ends the proof.
\end{proof}
\section{Distributional and classical solutions are caloric}\label{sec:distmvp}
Having obtained sufficient conditions for a caloric function to be a classical solution, the proof of the mean value property for distributional and classical solutions is rather standard, see, e.g., \cite{MR1671973}. We give a rather detailed outline for the sake of completeness, starting with an integrability result.
\begin{lemma}\label{lem:l1l1}
	If $u$ is a distributional solution to \eqref{eq:FHE}, then $u\in L^1((0,T), L^1(\mR^d,1\wedge |x|^{-d-\alpha}))$.
\end{lemma}
\begin{proof}
	There exist $r>0$ and $x_1,x_2\in D$ such that $B(x_1,2r)\cup B(x_2,2r)\subset D$ and $B(x_1,2r)\cap B(x_2,2r)=\emptyset$. For fixed $t\in (0,T)$ and for $i=1,2$, let $\phi_i\in C_c^\infty([0,t]\times D)$ be constant in $t$, nonpositive and satisfy $\phi_i\equiv -1$ in $B(x_i,r/2)$ and $\phi_i\equiv 0$ on $B(x_i,r)^c$. For $x\in B(x_i,r)^c$ we have
	\begin{align*}
		(\partial_t +\Delta^{\alpha/2})\phi_i(x) = \Delta^{\alpha/2}\phi_i(x) \geq \int_{B(x_i,r/2)} |x-z|^{-d-\alpha}\, dz \gtrsim 1\wedge |x|^{-d-\alpha}.
	\end{align*}
	Therefore, since the integrals in \eqref{eq:weakform} converge absolutely,
	\begin{align*}
		\infty > \int_0^t\int_{B(x_i,r)^c}	(\partial_t +\Delta^{\alpha/2})\phi_i(x) |u(s,x)|\, dx\, ds \gtrsim  \int_0^t\int_{B(x_i,r)^c} |u(s,x)|(1\wedge |x|^{-d-\alpha})\, dx\, ds.
	\end{align*}
	Thus, since $B(x_1,r)^c\cup B(x_2,r)^c = \mR^d$ we get the desired result.
\end{proof}

 Let $D_\eps = \{x\in D: \delta_D(x) > \eps\}$.
\begin{lemma}\label{lem:ueps}
	Assume that $u$ is a distributional solution to \eqref{eq:FHE}. Let $\eta_\eps\colon \mR\times \mR^d \to [0,\infty)$ be a space-time mollifier, i.e., a smooth function satisfying $\textstyle \int \eta_\eps = 1$ and $\eta_\eps = 0$ outside $(-\eps,\eps)\times B_\eps$. Then $u_\eps = u\ast \eta_\eps$ is a classical solution to \eqref{eq:FHE} on $(\eps,T-\eps)\times D_\eps$ with exterior data $g(t,x) = u_\eps(t,x)$ and initial data $u_0(x) = u_\eps(\eps,x)$.
	\end{lemma}
	\begin{proof}
		First, we will show that $u_\eps$ is a distributional solution in $(\eps,T-\eps)\times D_\eps$. To this end, let $\phi\in C_c^\infty([\eps,T-\eps)\times D_\eps)$ and define $\phi^{s,y}(t,x) = \phi(t+s,x+y)$. Note that for $(s,y)\in \supp\, \eta_\eps$, $\phi^{s,y}$ is a test function for $u$. By Tonelli's theorem and the fact that $u$ satisfies \eqref{eq:FHE}, for $t\in (\eps,t-\eps)$ we have
		\begin{align*}
			&\int_D \phi(t,x)u_\eps(t,x) \, dx - \int_D \phi(\eps,x)u_\eps(\eps,x)\, dx\\ = &\int_{-\eps}^\eps\int_{B_\eps}\eta_\eps(s,y) \bigg(\int_{D}\phi(t,x)u(t-s,x-y)\, dx - \int_D \phi(\eps,x)u(\eps-s,x-y)\, dx\bigg)\, dy\, ds\\ =&\int_{-\eps}^\eps\int_{B_\eps}\eta_\eps(s,y) \bigg(\int_{D}\phi^{s,y}(t-s,x)u(t-s,x)\, dx - \int_D \phi^{s,y}(\eps-s,x)u(\eps-s,x)\, dx\bigg)\, dy\, ds\\ =
			&\int_{-\eps}^\eps\int_{B_\eps}\eta_\eps(s,y) \int_{\eps - s}^{t-s} \int_{\mR^d} (\partial_t + \Delta^{\alpha/2})\phi^{s,x}(\tau,x)u(\tau,x)\, dx\, d\tau \, dy\, ds\\
			 =
			&\int_{\eps}^{t} \int_{\mR^d} (\partial_t + \Delta^{\alpha/2})\phi(\tau,x)u_\eps(\tau,x)\, dx\, d\tau.
		\end{align*}
		The fact that $u_\eps$ is a classical solution follows from integration by parts \cite[Lemma~3.3]{MR1671973}.
	\end{proof}
	\begin{theorem}\label{th:disttomvp}
		If $u$ is a distributional solution to \eqref{eq:FHE}, then it is almost everywhere equal to a function caloric in $[0,T)\times D$.
	\end{theorem}
	\begin{proof}
	 By Lemma~\ref{lem:ueps}, $u_\eps$ is a distributional and a classical solution to \eqref{eq:FHE} on $(\eps,T-\eps)\times D_\eps$ and by Lemma~\ref{lem:l1l1}, $u_\eps \in C^{\rm Dini}([\eps,T-\eps],L^1(1\wedge \nu))$. If $(t,x)\in (\eps,T-\eps)\times D_\eps$, we define $v_\eps(t,x) = \mE^{(t,x)}[u(\dot{X}_{\tau_{(\eps,T-\eps)\times D_\eps}})]$, otherwise we let $v_\eps = u_\eps$. By Theorem~\ref{th:mvpclass} we find that $v_\eps$ is a classical solution to \eqref{eq:FHE}. Since both $u_\eps$ and $v_\eps$ are continuous on $[\eps+\delta,T-\eps-\delta]\times \overline{D_{\eps+\delta}}$ for any $\delta>0$ and solve \eqref{eq:FHE} pointwise therein, the maximum principle implies that $u_\eps = v_\eps$ in $(\eps,T-\eps)\times D_\eps$. In particular, we get that $u_\eps$ is caloric in $(\eps,T-\eps)\times D_\eps$, so for every (small) $\delta>0$, 
	 \begin{align*}
	 	u_\eps(t,x) =  P_{t-\eps-\delta}^{D_{\eps+\delta}} u_\eps(\eps+\delta)(x) + \int_{\eps+\delta}^t \int_{D_{\eps+\delta}^c} J^{D_{\eps+\delta}}(t,x,s,z) u_\eps (s,z)\, dz \, ds,\quad t\in (\delta+\eps,T-\eps),\ x\in D_{\eps+\delta}.
	 \end{align*} By the definition of distributional solution and Lemma~\ref{lem:l1l1}, we find that for any $U\subset\subset D$ the norms $\|u(t)\|_{L^1(U)}$ are uniformly bounded locally in $(0,T)$. This implies that the norms $\|u_\eps(t)\|_{L^1(U)}$ are uniformly bounded locally in $(\eps,T-\eps)$ (for $\eps$ small enough). Thus, by the definition of distributional solution and Lemma~\ref{lem:l1l1}, we have $\textstyle \int_D u_\eps(t)\phi \to \int_D u(t)\phi$ for all $t\in(0,T)$ and $\phi\in C_c^\infty(D)$, and by approximation also for all $\phi\in C_c(D)$. In particular, for any $t_0>0$ and $t\in (t_0,T)$ we have $P_{t_0}^Uu_\eps(t)(x) \to P_{t_0}^U u(t)(x)$ as $\eps\to 0^+$. Therefore, for almost every $(t,x)$ such that $t\in(t_0,T)$ and $x\in U$, by Fatou's lemma we find that 
	 \begin{align*}
	 	\int_{t_0}^t \int_{U^c} J^U(t,x,s,z) u(s,z)\, dz\, ds \leq \lim\limits_{\eps\to 0^+} \int_{t_0}^t \int_{U^c} J^U(t,x,s,z) u_\eps(s,z)\, dz\, ds = u(t,x) - P_{t-t_0}^U u(t_0)(x) <\infty.
	 \end{align*}
	 By this and  \cite[Lemma~5.16]{HK}, it follows that for almost all $t\in (0,T)$,
	 \begin{align*}
	 	\int_{t_0}^t \int_{U^c} J^U(t,x,s,z) u(s,z)\, dz\, ds < \infty,\quad x\in U,
	 \end{align*}
	 from which we get that the integral is finite for all $t\in (0,T)$ and $x\in U$. Furthermore, by \cite[Proposition~5.15]{HK} the integrals are bounded independently of $x\in U$ and $t\in (t_0,T-\eps)$. By the dominated convergence theorem we therefore get that 
	 \begin{align*}
	 \lim\limits_{\eps\to 0^+} \int_{t_0}^t \int_{U^c} J^U(t,x,s,z) u_\eps(s,z)\, dz\, ds = \int_{t_0}^t \int_{U^c} J^U(t,x,s,z) u(s,z)\, dz\, ds.
	  \end{align*}
	  This implies that $u$ is almost everywhere equal to a function $v$ caloric in $(0,T)\times D$. By the definition of the distributional solution we get that $u(\eps,\cdot) \to u_0$ weakly in $L^1(U)$ for $U\subset\subset D$, therefore by the uniqueness of the representation in \cite[Theorem~6.5]{HK} we also get that $v = \mE^{(t,x)}u(\dot{X}_{\tau_{U\times (0,t)}})$ for all $x\in U$ and $t\in [0,T]$. This ends the proof.
	\end{proof}
\appendix
\section{Continuity of $\partial_t u$}
We will prove that under the assumptions of Theorem~\ref{th:mvpclass} the function $\partial_t u(t,x)$ given by \eqref{eq:dtu} is continuous in $(0,T)\times B_1$. Continuity in $x$ follows from estimates obtained in the proof of Theorem~\ref{th:mvpclass} and the dominated convergence theorem. It is also easy to see that the second term on the right-hand side of \eqref{eq:dtu} is continuous in $t$. It remains to show that the first term on the right-hand side of \eqref{eq:dtu} is continuous in $t$. To this end, we let $0<t_1<t_2<T$. Then, for fixed $x\in B_1$,
\begin{align}
	|u(t_1,x) - u(t_2,x)| &\leq \bigg| \int_{t_1}^{t_2}\int_{B_1^c} \partial_t J^{B_1}(t_2,x,s,z) (g(s,z) - g(t_2,z)) \, dz\, ds\bigg|\nonumber\\
	&+ \bigg| \int_0^{t_1}\int_{B_1^c} (\partial_t J^{B_1}(t_1,x,s,z) - \partial_t J^{B_1}(t_2,x,s,z))(g(s,z) - g(t_1,z)) \, dz\, ds \bigg|\label{eq:second}\\
	&+ \bigg|\int_0^{t_1} \int_{B_1^c} \partial_t J^{B_1}(t_2,x,s,z)(g(t_1,z) - g(t_2,z))\, dz\, ds\bigg|.\nonumber
\end{align}
By Dini continuity of $g$ the first integral converges to 0 as $t_2 - t_1 \to 0$. In the last integral we use Fubini--Tonelli and the fundamental theorem of calculus:
\begin{align*}
	 &\bigg|\int_0^{t_1} \int_{B_1^c} \partial_t J^{B_1}(t_2,x,s,z)(g(t_1,z) - g(t_2,z))\, dz\, ds\bigg|\\
	 &\leq \int_{B_1^c} |g(t_1,z) - g(t_2,z)||J^{B_1}(t_2,x,t_1,z) - J^{B_1}(t_2,x,0,z)|\, dz\\
	 &\leq \int_{B_1^c} |g(t_1,z) - g(t_2,z)|(J^{B_1}(t_2,x,t_1,z) + J^{B_1}(t_2,x,0,z))\, dz.
\end{align*}
Since $u$ is continuous in $(0,T)\times B_2$ and $u\in C([0,T],L^1(1\wedge \nu))$, the above expression converges to 0 as $t_2 - t_1 \to 0$. In the integral \eqref{eq:second} we note that the integral over $B_1^c$ converges to 0 as $t_2 - t_1 \to 0$ for each fixed $s$. Furthermore, by \eqref{eq:JBest} and by the Dini continuity of $g$, the family of functions $\{\phi_{t_1,t_2}: 0<t_1\leq t_2 < T\}$ given by
\begin{align*}
	\phi_{t_1,t_2}(s) =  \textbf{1}_{(0,t_1)}(s)\int_{B_1^c} |\partial_t J^{B_1}(t_1,x,s,z) - \partial_t J^{B_1}(t_2,x,s,z)||g(s,z) - g(t_1,z)| \, dz,
\end{align*}
is uniformly integrable. Therefore, by Vitali's theorem the integral in \eqref{eq:second} converges to 0 as $t_2 - t_1 \to 0$. This proves that $u$ is continuous in $t$.
\section*{Declarations}
\subsection*{Ethical Approval} Not applicable.

\subsection*{Funding} Research was partially supported by the National Science Center of Poland (NCN) under grant 2019/35/N/ST1/04450.

\subsection*{Availability of data and materials} Not applicable.
\bibliographystyle{abbrv}
\bibliography{HK.bib}

\begin{thebibliography}{10}

\bibitem{MR3393247}
N.~Abatangelo.
\newblock Large {$s$}-harmonic functions and boundary blow-up solutions for the
  fractional {L}aplacian.
\newblock {\em Discrete Contin. Dyn. Syst.}, 35(12):5555--5607, 2015.

\bibitem{MR4635724}
G.~Armstrong, K.~Bogdan, T.~Grzywny, {\L{}}.~Le\.{z}aj, and L.~Wang.
\newblock Yaglom limit for unimodal {L}\'{e}vy processes.
\newblock {\em Ann. Inst. Henri Poincar\'{e} Probab. Stat.}, 59(3):1688--1721,
  2023.

\bibitem{HK}
G.~Armstrong, K.~Bogdan, and A.~Rutkowski.
\newblock Caloric functions and boundary regularity for the fractional
  {L}aplacian in {L}ipschitz open sets.
\newblock {\em Math. Ann.}, 391(1):1199--1252, 2025.

\bibitem{MR3211862}
B.~Barrios, I.~Peral, F.~Soria, and E.~Valdinoci.
\newblock A {W}idder's type theorem for the heat equation with nonlocal
  diffusion.
\newblock {\em Arch. Ration. Mech. Anal.}, 213(2):629--650, 2014.

\bibitem{MR1881259}
P.~Biler, G.~Karch, and W.~A. Woyczy\'{n}ski.
\newblock Asymptotics for conservation laws involving {L}\'{e}vy diffusion
  generators.
\newblock {\em Studia Math.}, 148(2):171--192, 2001.

\bibitem{MR0119247}
R.~M. Blumenthal and R.~K. Getoor.
\newblock Some theorems on stable processes.
\newblock {\em Trans. Amer. Math. Soc.}, 95:263--273, 1960.

\bibitem{MR1704245}
K.~Bogdan.
\newblock Representation of {$\alpha$}-harmonic functions in {L}ipschitz
  domains.
\newblock {\em Hiroshima Math. J.}, 29(2):227--243, 1999.

\bibitem{MR1671973}
K.~Bogdan and T.~Byczkowski.
\newblock Potential theory for the {$\alpha$}-stable {S}chr\"odinger operator
  on bounded {L}ipschitz domains.
\newblock {\em Studia Math.}, 133(1):53--92, 1999.

\bibitem{MR2569321}
K.~Bogdan, T.~Byczkowski, T.~Kulczycki, M.~Ryznar, R.~Song, and
  Z.~Vondra\v{c}ek.
\newblock {\em Potential analysis of stable processes and its extensions},
  volume 1980 of {\em Lecture Notes in Mathematics}.
\newblock Springer-Verlag, Berlin, 2009.
\newblock Edited by Piotr Graczyk and Andrzej Stos.

\bibitem{MR2722789}
K.~Bogdan, T.~Grzywny, and M.~Ryznar.
\newblock Heat kernel estimates for the fractional {L}aplacian with {D}irichlet
  conditions.
\newblock {\em Ann. Probab.}, 38(5):1901--1923, 2010.

\bibitem{MR2283957}
K.~Bogdan and T.~Jakubowski.
\newblock Estimates of heat kernel of fractional {L}aplacian perturbed by
  gradient operators.
\newblock {\em Comm. Math. Phys.}, 271(1):179--198, 2007.

\bibitem{bogdan2018}
K.~Bogdan, Z.~Palmowski, and L.~Wang.
\newblock Yaglom limit for stable processes in cones.
\newblock {\em Electron. J. Probab.}, 23:Paper No. 11, 19, 2018.

\bibitem{MR3737628}
K.~Bogdan, J.~Rosi\'{n}ski, G.~Serafin, and {\L}.~Wojciechowski.
\newblock L\'{e}vy systems and moment formulas for mixed {P}oisson integrals.
\newblock In {\em Stochastic analysis and related topics}, volume~72 of {\em
  Progr. Probab.}, pages 139--164. Birkh\"{a}user/Springer, Cham, 2017.

\bibitem{MR521856}
C.~C. Burch.
\newblock The {D}ini condition and regularity of weak solutions of elliptic
  equations.
\newblock {\em J. Differential Equations}, 30(3):308--323, 1978.

\bibitem{MR4462819}
H.~Chan, D.~G\'{o}mez-Castro, and J.~L. V\'{a}zquez.
\newblock Singular solutions for fractional parabolic boundary value problems.
\newblock {\em Rev. R. Acad. Cienc. Exactas F\'{\i}s. Nat. Ser. A Mat. RACSAM},
  116(4):Paper No. 159, 38, 2022.

\bibitem{MR3115838}
H.~Chang-Lara and G.~D\'{a}vila.
\newblock Regularity for solutions of nonlocal parabolic equations {II}.
\newblock {\em J. Differential Equations}, 256(1):130--156, 2014.

\bibitem{MR2515419}
Z.-Q. Chen.
\newblock On notions of harmonicity.
\newblock {\em Proc. Amer. Math. Soc.}, 137(10):3497--3510, 2009.

\bibitem{MR2677618}
Z.-Q. Chen, P.~Kim, and R.~Song.
\newblock Heat kernel estimates for the {D}irichlet fractional {L}aplacian.
\newblock {\em J. Eur. Math. Soc. (JEMS)}, 12(5):1307--1329, 2010.

\bibitem{MR2008600}
Z.-Q. Chen and T.~Kumagai.
\newblock Heat kernel estimates for stable-like processes on {$d$}-sets.
\newblock {\em Stochastic Process. Appl.}, 108(1):27--62, 2003.

\bibitem{Chen1998}
Z.-Q. Chen and R.~Song.
\newblock Estimates on {G}reen functions and {P}oisson kernels for symmetric
  stable processes.
\newblock {\em Math. Ann.}, 312(3):465--501, 1998.

\bibitem{MR1329992}
K.~L. Chung and Z.~X. Zhao.
\newblock {\em From {B}rownian motion to {S}chr\"odinger's equation}, volume
  312 of {\em Grundlehren der Mathematischen Wissenschaften}.
\newblock Springer-Verlag, Berlin, 1995.

\bibitem{MR1814344}
J.~L. Doob.
\newblock {\em Classical potential theory and its probabilistic counterpart}.
\newblock Classics in Mathematics. Springer-Verlag, Berlin, 2001.
\newblock Reprint of the 1984 edition.

\bibitem{MR3278937}
N.~Eldredge and L.~Saloff-Coste.
\newblock Widder's representation theorem for symmetric local {D}irichlet
  spaces.
\newblock {\em J. Theoret. Probab.}, 27(4):1178--1212, 2014.

\bibitem{MR709573}
E.~Fabes and S.~Salsa.
\newblock Estimates of caloric measure and the initial-{D}irichlet problem for
  the heat equation in {L}ipschitz cylinders.
\newblock {\em Trans. Amer. Math. Soc.}, 279(2):635--650, 1983.

\bibitem{MR3462074}
X.~Fern\'{a}ndez-Real and X.~Ros-Oton.
\newblock Boundary regularity for the fractional heat equation.
\newblock {\em Rev. R. Acad. Cienc. Exactas F\'{\i}s. Nat. Ser. A Mat. RACSAM},
  110(1):49--64, 2016.

\bibitem{MR833742}
M.~Freidlin.
\newblock {\em Functional integration and partial differential equations},
  volume 109 of {\em Annals of Mathematics Studies}.
\newblock Princeton University Press, Princeton, NJ, 1985.

\bibitem{MR4194536}
T.~Grzywny, M.~Kassmann, and {\L{}}.~Le\.{z}aj.
\newblock Remarks on the nonlocal {D}irichlet problem.
\newblock {\em Potential Anal.}, 54(1):119--151, 2021.

\bibitem{MR4668351}
K.~M. Hui and K.-S. Chou.
\newblock Nonnegative solutions of the heat equation in a cylindrical domain
  and {W}idder's theorem.
\newblock {\em J. Math. Anal. Appl.}, 532(2):Paper No. 127929, 18, 2024.

\bibitem{ikeda1962}
N.~Ikeda and S.~Watanabe.
\newblock On some relations between the harmonic measure and the {L}évy
  measure for a certain class of {M}arkov processes.
\newblock {\em J. Math. Kyoto Univ.}, 2(1):79--95, 1962.

\bibitem{MR1991120}
T.~Jakubowski.
\newblock The estimates for the {G}reen function in {L}ipschitz domains for the
  symmetric stable processes.
\newblock {\em Probab. Math. Statist.}, 22(2, Acta Univ. Wratislav. No.
  2470):419--441, 2002.

\bibitem{MR1643611}
T.~Kulczycki.
\newblock Intrinsic ultracontractivity for symmetric stable processes.
\newblock {\em Bull. Polish Acad. Sci. Math.}, 46(3):325--334, 1998.

\bibitem{MR3767143}
T.~Kulczycki and M.~Ryznar.
\newblock Gradient estimates of {D}irichlet heat kernels for unimodal
  {L}\'{e}vy processes.
\newblock {\em Math. Nachr.}, 291(2-3):374--397, 2018.

\bibitem{MR350027}
N.~S. Landkof.
\newblock {\em Foundations of modern potential theory}.
\newblock Die Grundlehren der mathematischen Wissenschaften, Band 180.
  Springer-Verlag, New York-Heidelberg, 1972.
\newblock Translated from the Russian by A. P. Doohovskoy.

\bibitem{MR2139572}
L.~Riahi.
\newblock Comparison of {G}reen functions and harmonic measures for parabolic
  operators.
\newblock {\em Potential Anal.}, 23(4):381--402, 2005.

\bibitem{MR1739520}
K.~Sato.
\newblock {\em L\'evy processes and infinitely divisible distributions},
  volume~68 of {\em Cambridge Studies in Advanced Mathematics}.
\newblock Cambridge University Press, Cambridge, 1999.
\newblock Translated from the 1990 Japanese original, Revised by the author.

\bibitem{MR2907452}
N.~A. Watson.
\newblock {\em Introduction to heat potential theory}, volume 182 of {\em
  Mathematical Surveys and Monographs}.
\newblock American Mathematical Society, Providence, RI, 2012.

\end{thebibliography}
\end{document}